\documentclass[12pt]{article}

\usepackage{amsmath}
\usepackage{amsfonts}
\usepackage{amsthm}
\usepackage{mathrsfs}
\usepackage{yfonts} 
\usepackage{bm}
\usepackage{bbm}
\usepackage{hyperref}
\usepackage{mathtools}
\usepackage{geometry}

\DeclarePairedDelimiter\floor{\lfloor}{\rfloor}
\makeatletter 
\@addtoreset{equation}{section}
\makeatother  

\newtheoremstyle{mystyle}
{5pt} 
{5pt} 
{\itshape} 
{\parindent} 
{\itshape} 
{:} 
{.5em} 
{} 
\theoremstyle{mystyle}

\newtheorem{definition}{Definition}[section]
\newtheorem{assumption}[definition]{Assumption}
\newtheorem{thm}{Theorem}[section]
\newtheorem{lemma}[thm]{Lemma}
\newtheorem{prop}[thm]{Proposition}
\newtheorem{col}[thm]{Corollary}
\newtheorem{remark}[thm]{Remark}

\newcommand{\norm}[1]{\left\lVert#1\right\rVert}
\title{Models of Gradient Type \\with Sub-Quadratic Actions}
\author{By Zichun Ye \\
	\itshape{University of British Columbia}
}

\begin{document}
	\maketitle
	
	\pagenumbering{arabic}

	\pagenumbering{arabic}
	\abstract{We consider models of gradient type, which are the densities of a collection of real-valued random variables $\phi :=\{\phi_x: x \in \Lambda\}$ given by $Z^{-1}\exp({-\sum\nolimits_{j \sim k}V(\phi_j-\phi_k)})$. We focus our study on the case that $V(\nabla\phi) = [1+(\nabla\phi)^2]^\alpha$ with $0 < \alpha < 1/2$, which is a non-convex potential. We introduce an auxiliary field $t_{jk}$ for each edge and represent the model as the marginal of a model with log-concave density. Based on this method, we prove that finite moments of the fields $\left<[v \cdot \phi]^p \right>$ are bounded uniformly in the volume. This leads to the existence of infinite volume measures. Also, every translation invariant, ergodic
	infinite volume Gibbs measure for the potential $V$ above scales to a Gaussian free field.
	}

\section{Introduction}
\hspace*{\parindent} 
	The models of gradient type are a class of models arising in equilibrium statistical mechanics, like random interface model for phase coexistence \cite{Fun03}.    
	For a finite subset $\Lambda$ of $\mathbb{Z}^d$, a model of gradient type is a collection of real-valued random variables $\phi :=\{\phi_x: x \in \Lambda\}$ with distribution given by
	\begin{equation}\label{generalmodel}
		\frac{1}{Z}\exp\left({-\sum\limits_{jk \in E(\Lambda)}V(\phi_j-\phi_k)}\right)\prod_{x \in \Lambda}d\phi_x,
	\end{equation}
	where $E(\Lambda)$ is the set of unordered pairs of nearest neighbors in $\Lambda$, and $d\phi_x$ is the Lebesgue measure and $V$ is an even, measurable function, called the potential function.
	It should be noticed that the Hamiltonian $H_\Lambda := \sum_{jk}V(\phi_j-\phi_k)$ does not change under the translation $\phi \to \phi+c$ where c is a constant, so boundary conditions are introduced to break this symmetry which would otherwise make it impossible to normalize the distribution. 
	
	In \cite{FS97}, Funaki and Spohn used this model to study the interface with energy $H_\Lambda$ and made it popular in the research of probability theory. Also see  \cite{DGI00} and \cite{Gia02} for more detail about physical interpretations of these models.
	The study by  Funaki and Spohn \cite{FS97} required strict convexity of $V$, namely:
	\begin{equation}\label{strictcon}
		c_- \le V''(x) \le c_+, x \in \mathbb{R}, \mathrm{for~some~} c_-, c_+ > 0.
	\end{equation}
	One example of this kind of $V$ is the quadratic function, under which the measure (\ref{generalmodel}) is Gaussian and the model becomes the massless free field. The condition of strict convexity on $V$ makes the model tractable because it implies the Brascamp- Lieb bound given in Theorem \ref{BLBthm} \cite{BL76} and the H-S random walk representation of the model \cite{DGI00}. 
	In \cite{CD12,CDM09}, C. Cotar et al. gave results for non-convex $V$. They studied a class of models with $V$ admitting the representation
	\begin{equation}\label{ccv}
		V(t) = V_0(t)+g_0(t)
	\end{equation}
	where $V_0$ satisfies (\ref{strictcon}) and $g_0 \in C^2(\mathbb{R})$ has a negative bounded second derivative. In \cite{BK07,BS11}, Marek Biskup, Roman Koteck{\'y} and Herbert Spohn gave results about a class of models with $V$ admitting the representation
	\begin{equation}\label{biskup}
		e^{-V(t)} = \int  \varrho (d\kappa) \exp\big[-\frac{1}{2}\kappa t^2 \big]
	\end{equation}
	where $\varrho$ is a positive measure with compact support in $(0, \infty)$. This class of $V$ is symmetric, but non-convex in general, for example, when $\varrho = p\delta_a+(1-p)\delta_b$ is two-point measure.
	
	In this paper, we discuss a special case with a non-convex potential function $V$. We focus our study on the special case of the measure defined by (\ref{generalmodel}) with
	\begin{equation}
		V(\nabla\phi) = [1+\beta(\nabla\phi)^2]^\alpha
	\end{equation}
	with  $0 < \alpha < \frac{1}{2}$. In this case, $V$ is not convex, while $V$ is convex when $\alpha > \frac{1}{2}$. We will follow the strategy of Biskup and Spohn. To be specific, we write $e^{-V}$ in the form (\ref{biskup}) as
	\begin{equation}
		\int e^{-\kappa[1+\beta(\nabla \phi)^2]}f(\kappa)d\kappa = e^{-V(\nabla \phi)},
	\end{equation}
	and then the model can be represented as
	\begin{equation}\label{expandform}
		\frac{1}{Z}e^{-\sum [1+\beta(\nabla \phi)^2]e^{t_{jk}}}f(e^{t_{jk}})e^{t_{jk}}\prod_{i \in \Lambda}d\phi_i\prod_{jk \in E}dt_{jk}.
	\end{equation}
	Here we make the substitution $\kappa \mapsto e^t$ so that the density has the nice property of log concavity.
	
	We are interested in this class of $V$ for several reasons. Firstly,  this class of measures has features in common with the massless free field, namely Markov property and Osterwalder Schrader positivity \cite{FIL78phase}. Secondly, as the model studied in \cite{BS11} has a phase transition, this model could also have phase transitions as they have a similar structure. Last but not least, when $\alpha = 1/2$, as studied in \cite{BS12}, the Hamiltonian  is similar to the one appearing in the limiting measure describing linearly edge reinforced random walk.

This paper is composed of two parts. In the first part, models of gradient type are precisely introduced. We consider the field $\phi$ on the sites of the lattice, where $\phi$ is called the height variables in latter context.  This system is called Gibbs measure. As the main result of this part, we will prove that the moments of the finite volume measure satisfy a uniform bound. As a result, we will prove that the infinite volume Gibbs measure exists for $d \ge 3$. Here we extend the proof of \cite{BS12} for the case $\alpha = 1/2$ to $\alpha \in (1, 1/2)$.
	
 In the second part, we are interested in the scaling limit of the models. We will show that every translation-invariant, ergodic infinite volume measures scales to a Gaussian free field. The proof is also based on an integral representation of our $V$ similar to (\ref{biskup}). Then we will represent our infinite volume measure as a mixture over Gaussian gradient measures with a random coupling constant $\omega_{jk}$ for each edge $jk$. Every individual Gaussian gradient measures has covariance being the inverse of the minus of the operator
	\begin{equation}
		(\mathcal{L}^\omega_Xf)(x) = \sum_{y \sim x}\omega(x, y)(f(y)-f(x)).
	\end{equation}
	Then the Gaussian measure can be analyzed by invoking a random walk representation; $\mathcal{L}^\omega_X$ is the generator of a random walk with symmetric random jump rates, known, equivalently, as a random conductance model. As a sketch of the proof, we will firstly extend the infinite volume measure to a new measure on the space of both fields $\phi$ and weights $\omega$. Then we will give the random walk representation of the mean and covariance matrix of infinite volume. Lastly, with the help of functional central limit theory (invariance principle in \cite{ADS14}) and heat kernel bound (in \cite{ADS15H}) of random conductance model, we prove the results of scaling limit. 

\section{Main Results}
\subsection{Bounds of the moments}
\hspace*{\parindent} We are interested in a hyper-surface embedded in $d + 1$ dimensional space $\mathbb{R}^{d+1}$. The hyper-surface is represented by a graph viewed from a fixed reference $d$-dimensional hyperplane $\Gamma$ located in the space $\mathbb{R}^{d+1}$. In other words, there are no overhangs and the location of the hyper-surface is described by the configuration $\phi = \{\phi(x) \in \mathbb{R}; x \in \Gamma\}$, where $\phi_x$ is the height of the hyper-surface above $x$. The hyperspace $\Gamma$ is discretized and taken as $\Gamma = \Lambda \subset \mathbb{Z}^d$. The height $\phi_x$ is not discretized. Let $\Omega_\Lambda = \mathbb{R}^{\Lambda}$ be the set of all configurations over $\Lambda$ and $\Omega = \Omega_{\mathbb{Z}^d}$ be the configurations over $\mathbb{Z}^d$.

We think of $\mathbb{Z}^d$ as a graph with edges
\begin{equation}
E = \Big\{\{j,k\} : j, k \in \mathbb{Z}^d, \norm{j-k}_2 = 1 \Big\}
\end{equation}
where $\norm{\cdot}_2$ is the Euclidean norm. We use the notation $jk = \{j, k\}$ for the undirected edges.
For a finite subset $\Lambda$ of $\mathbb{Z}^d$, let $E(\Lambda)$ be the set of edges with at least one vertex in $\Lambda$, namely,
\begin{equation}
E(\Lambda)=\{jk \in E | jk \cap \Lambda \not=\emptyset\}.
\end{equation} 
Given a finite set $\Lambda \subset \mathbb{Z}^d$ and the potential function $V(x) = [1+\beta x^2]^\alpha$, we define the Hamiltonian $H_\Lambda$ by
\begin{equation}\label{Hamia}
H_{\Lambda}(\phi) := \sum_{jk \in E(\Lambda)} V(\phi_j-\phi_k) = \sum_{jk \in E(\Lambda)} \left(\left(1+\beta(\phi_j-\phi_k\right)^2\right)^\alpha.
\end{equation}
Then for a boundary condition $\psi \in \Omega$, we define the finite volume Gibbs measure over $\Lambda$ given by
\begin{equation}\label{modela}
\mu^{\psi}_{\Lambda} := \frac{1}{Z_{\Lambda}(\psi_{\Lambda^c})}e^{-H_{\Lambda}(\phi\vee\psi)}\prod_{j \in \Lambda}d\phi_j .
\end{equation}
Here $(\phi\vee\psi)(k) = \phi_k$ for $k \in \Lambda$ and $=\psi_k$ for $k\in \Lambda^c$, and $Z_{\Lambda}(\psi_{\Lambda^c})$ is the normalization constant defined by
\begin{equation}
Z_{\Lambda}(\psi_{\Lambda^c}) = \int_{\mathbb{R}^\Lambda} e^{-H_{\Lambda}(\phi)}\prod_{j \in \Lambda}d\phi_j . 
\end{equation}

The Hamiltonian $H$ defined in (\ref{Hamia}) is called the massless Hamiltonian. The massive Hamiltonian is given by 
\begin{equation}\label{defHm}
H_{\Lambda, \epsilon}(\phi) = H_{\Lambda }(\phi)+\epsilon\sum_{j \in \Lambda} \phi_j^2,
\end{equation}
where $\epsilon > 0$.
This provides an approximation to the massless Hamiltonian when letting $\epsilon \to 0$. In the massive case, we introduce the periodic boundary condition. Let $\Lambda$ be a cuboid in $\mathbb{Z}^d$ with $N = (N_1, \cdots, N_d)$ being the side length. Define the graph $\mathbb{T}_N$ such that the vertex set $V(\mathbb{T}_N) = \Lambda$ and the edge set $E(\mathbb{T}_N)$ is the union of the edges of $\Lambda$ and the set of additional edges incident to vertices $x, y$ on opposite boundaries such that $y_i - x_i = N_i$ for some $i = 1, \cdots, d$. For $\phi \in \Omega_\Lambda$, we define the massive torus Hamiltonian $H^p_{\Lambda,  \epsilon}$ by
\begin{equation}
H^p_{\Lambda,  \epsilon}(\phi) = \sum_{jk \in E(\mathbb{T}_N)} V(\phi_j-\phi_k)+\epsilon\sum_{j \in \mathbb{T}_N} \phi_j^2
\end{equation}
where $E(\mathbb{T}_N)$ is the set of all edges in $\mathbb{T}_N$. Then the finite volume measure on $\mathbb{R}^\Lambda$ with periodic boundary condition is defined as 
\begin{equation}\label{model}
\mu^p_{\Lambda, \epsilon} := \frac{1}{Z_{\Lambda,\epsilon}}e^{-H^p_{\Lambda,  \epsilon}(\phi)}\prod_{j \in \Lambda}d\phi_j.
\end{equation}
where $\epsilon >  0$.

Now we introduce an auxiliary field $t$ on the edges. For $\Lambda$ and a boundary condition $\psi$, we define a Gaussian action
\begin{equation}\label{Gaction}
A^\psi(\phi, t) = \sum_{jk \in E(\Lambda)}\Big(1+\beta(\phi\vee\psi(j)-\phi\vee\psi(k))^2\Big)e^{t_{jk}}+\sum_{j \in \Lambda}\epsilon \phi_j^2, \quad\quad  \epsilon \ge 0.
\end{equation} 
Let $f_\alpha(x)$ be the unique positive density such that
\begin{equation}\label{lpstable}
\int_{0}^{\infty} e^{-\lambda x}f_\alpha(x) dx = e^{-\lambda^\alpha}.
\end{equation}
The existence, uniqueness and positivity of $f_\alpha(x)$ are proved in Chapter IX, section 11 in \cite{Yosida80}.  Consider the measure $\hat{\mu}^\psi_{\Lambda,\epsilon}$ on $\mathbb{R}^\Lambda\times\mathbb{R}^{E(\Lambda)}$ defined by
\begin{equation}\label{defextmeasure}
\hat{\mu}^\psi_{\Lambda,\epsilon}(d\phi, dt) := \frac{1}{Z^\psi_{\Lambda,\epsilon}(\alpha)} e^{-A^\psi(\phi, t)}\prod_{jk \in E(\Lambda)}\Big(f_\alpha(e^{t_{jk}})e^{t_{jk}}dt_{jk}\Big)\prod_{j \in \Lambda}d\phi_j,
\end{equation}
where the partition function $Z^\psi_{\Lambda,\epsilon}(\alpha)$ is defined to be
\begin{eqnarray}
Z^\psi_{\Lambda,\epsilon}(\alpha)
& = & \int e^{-A(\phi, t)}\prod_{jk \in E(\Lambda)}\Big(f_\alpha(e^{t_{jk}})e^{t_{jk}}dt_{jk}\Big)\prod_{j \in \Lambda}d\phi_j \label{ameasure}\\
&=&\int e^{-\sum\nolimits_{jk \in E(\Lambda)
	}\big(1+\beta (\phi_j-\phi_k)^2\big)^\alpha-\sum\nolimits_{j\in\Lambda}\epsilon\phi_j^2}\prod_{j \in \Lambda}d\phi_j. \label{withoutt}
\end{eqnarray}
By integrating all $t_{jk}$ over $\mathbb{R}$ we get (\ref{withoutt}) from (\ref{ameasure}).  Thus the $\phi$ marginal of $\hat{\mu}^\psi_{\Lambda, \epsilon}$ is the massive finite volume Gibbs measure $\mu^\psi_{\Lambda,\epsilon}$ with the Hamiltonian given in (\ref{defHm}).

For $\Lambda$ be a cuboid in $\mathbb{Z}^d$ with side length $N = (N_1, \cdots, N_d)$, we define the Gaussian action with periodic boundary condition by
\begin{equation}\label{Gactionperiodic}
A^p(\phi, t) = \sum_{jk \in E(\mathbb{T}_N)}\Big(1+\beta(\phi(j)-\phi(k))^2\Big)e^{t_{jk}}+\sum_{j \in \Lambda}\epsilon \phi_j^2, \quad\quad  \epsilon \ge 0.
\end{equation}
Similarly we define the measure $\hat{\mu}^p_{\Lambda,\epsilon}$ on $\mathbb{R}^{\mathbb{T}_N}\times\mathbb{R}^{E(\mathbb{T}_N)}$ defined by
\begin{equation}\label{defextmeasureperiodic}
\hat{\mu}^p_{\Lambda,\epsilon}(d\phi, dt) := \frac{1}{Z^p_{\Lambda,\epsilon}(\alpha)} e^{-A^p(\phi, t)}\prod_{jk \in E(\mathbb{T}_N)}\Big(f_\alpha(e^{t_{jk}})e^{t_{jk}}dt_{jk}\Big)\prod_{j \in \Lambda}d\phi_j,
\end{equation}
where $Z^p_{\Lambda,\epsilon}(\alpha)$ is the partition function. Then $\phi$ marginal of $\hat{\mu}^p_{\Lambda, \epsilon}$ is $\mu^p_{\Lambda,\epsilon}$ in (\ref{model}).

To state the $t$ marginal of $\hat{\mu}^p_{\Lambda, \epsilon}$, let $[v; w] = \sum_jv_jw_j$ be the usual scalar product in $\mathbb{R}^\Lambda$, and write $[v; w] = v \cdot w$ for short. Then $D_{\Lambda, \epsilon}(t)$ is defined  by the quadratic form
\begin{equation}\label{defD}
[f; D_{\Lambda,\epsilon}(t)f] = \sum_{jk\in E(\mathbb{T}_N)}\beta(f_j - f_k)^2e^{t_{jk}}+\epsilon\sum_{j \in \Lambda} f_j^2
\end{equation}
for all $f: \Lambda \to \mathbb{R}$ satisfying periodic boundary condition. The eigenvalues of $D_{\Lambda,\epsilon}(t)$ are positive because $[f; D_{\Lambda,\epsilon}(t)f] > 0$ for $f\not=0$, and thus $D_{\Lambda,\epsilon}(t)$ is invertible in $\mathbb{R}^\Lambda$. Let $G_{\Lambda,\epsilon}(t) = (D_{\Lambda,\epsilon}(t))^{-1}$ be the Green's function.
By evaluating the integral over $\phi$ over $\mathbb{R}^\Lambda$ in (\ref{defextmeasure}), we have the $t$ marginal $\hat{\mu}^p_{\Lambda, \epsilon}$
\begin{equation}
\hat{\mu}^p_{\Lambda, \epsilon}(dt) = \frac{1}{Z'_{\Lambda,\epsilon}(\alpha)}\det[D_{\Lambda,\epsilon}(t)]^{-1/2}\prod_{jk \in E(\mathbb{T}_N)}\Big(\exp(-e^{t_{jk}}+t_{jk})f_\alpha(e^{t_{jk}})dt_{jk}\Big),\label{modelgeneral}
\end{equation}
where $Z'_{\Lambda,\epsilon}$ is the partition function for $t$ variable defined by
\begin{equation}
Z'_{\Lambda, \epsilon}(\alpha) = \int \det[D_{\Lambda,\epsilon}(t)]^{-1/2}\prod_{jk \in E(\mathbb{T}_N)}\Big(\exp(-e^{t_{jk}}+t_{jk})f_\alpha(e^{t_{jk}})dt_{jk}\Big).\label{modelgeneral1}
\end{equation}
Here we absorb the factors of $\pi$ into the new partition function $Z'_{\Lambda, \epsilon}$.

Let $\alpha < \frac{1}{2}$ and $\left<\cdot\right>_{\alpha, \Lambda,\epsilon}$ denote the expectation in $t$ and $\phi$ with respect to the measure $\hat{\mu}^p_{\Lambda, \epsilon}$ in (\ref{defextmeasureperiodic}).  
Our first result is the bounds on $\left<e^{\lambda t}\right>_{\alpha, \Lambda,\epsilon}$ for all $\lambda \in \mathbb{R}$. Notice that all these bounds are uniform in $\epsilon$ and $\Lambda$ and holds for all $d \ge 1$.

\begin{prop} \label{propbda}
	Assume $\alpha \in (0, \frac{1}{2})$. For $\lambda \in \mathbb{R}$, there is a constant $C(\lambda, \alpha)$ such that for all $jk \in E$,
	\begin{equation}\label{tbd}
	\left<e^{\lambda t_{jk}}\right>_{\alpha, \Lambda,\epsilon} \le C(\lambda, \alpha).
	\end{equation}
\end{prop}

From Proposition \ref{propbda}, we prove the bound for the finite order moment of the $\phi$ field. To state the theorem, define the lattice Laplacian $-\Delta^p_\Lambda$ with periodic boundary condition by 
\begin{equation}\label{deflaplace}
[\phi; -\Delta^p_\Lambda\phi] = \sum\limits_{j, k \in \Lambda, jk\in E}(\phi_j-\phi_k)^2
\end{equation}
for all $\phi : \Lambda \to \mathbb{R}$ satisfying periodic boundary condition. However, $-\Delta^p_\Lambda$ is not invertible as 0 is an eigenvalue of $-\Delta^p_\Lambda$ with eigenvector $\phi \equiv \text{const}$ on $\Lambda$. Thus we consider the domain $\mathcal{D}_p = \left\{v \in \mathbb{R}^\Lambda |~ v \text{ compact support and } \right [v; 1] = 0\}$ for the test function $v$ and let $G^p_{\Lambda} = (-\Delta^p_\Lambda)^{-1}$ be the inverse of the lattice Laplacian on the domain $\mathcal{D}_p$.
\begin{thm}\label{mymain1}
	Assume $\alpha \in (0, \frac{1}{2})$. If $v \in \mathcal{D}_p$, then 
	\begin{equation}
	\left<(\phi \cdot v)^{2p}\right>_{\alpha, \Lambda,\epsilon} \le \tilde{C}(\alpha,p)[v;G_{\Lambda}^pv]^{n}
	\end{equation}
\end{thm}	

The bounds for the finite order moments provide tightness and furthermore give the existence of infinite volume Gibbs measure defined as follows.

\begin{definition}[DLR-equation]\label{defDLR}
	Let $\mathscr{F}_{\Lambda} = \sigma\left({\phi_j, j \in \Lambda}\right)$. We say a measure $\mu$ on $(\Omega, \mathscr{F}_{\mathbb{Z}^d})$ is an infinite volume Gibbs measure for parameter $\beta$ if for any finite set $\Lambda \subset \mathbb{Z}^d$ and $\mu-a.s.~\psi$, the following DLR-equation holds:
	\begin{equation}\label{DLR}
	\mu(~\cdot~| \mathscr{F}_{\Lambda^c})(\psi) = \mu^{\psi}_{\Lambda, \beta}(\cdot),
	\end{equation}
	where $\mu^{\psi}_{\Lambda, \beta}(\cdot)$ is defined in (\ref{modela}). Let $\mathscr{G}$ denote the set of all infinite volume Gibbs measures.
\end{definition}

\begin{thm}[Existence of infinite volume Gibbs measure]\label{existGibbs}
	Assume $\alpha \in (0, \frac{1}{2})$. For $d \ge 3$, there exists a translation invariant, ergodic Gibbs measure $\mu$ on $(\Omega, \mathscr{F}_{\mathbb{Z}^d})$.
\end{thm}

\subsection{Scaling limit}
\hspace*{\parindent}In the rest of this section, we will define of Gaussian free field and state our result for the scaling limit. We will follow the setting and notation of section 2.2 of \cite{BS11}. 

Firstly we need the definition of the function space we will work on. Let $\Delta_\mathbb{R}$ denote the Laplace differential operator in $\mathbb{R}^d$ and consider the set
\begin{equation}
\mathcal{H}_0 := \{\Delta_\mathbb{R} g : g \in C^\infty_0(\mathbb{R}^d) \}.
\end{equation}
Then we define the norm on $\mathcal{H}_0$
\begin{equation}
\norm{f}_\mathcal{H} := [(f, f) + (f, -\Delta_\mathbb{R}^{-1}f)]^{1/2}
\end{equation}
where $(\cdot, \cdot)$ is the usual inner product in $L^2$-space. Let $\mathcal{H}$ be the completion of  $\mathcal{H}_0$ in this norm.

With the space $\mathcal{H}$, we will give the following  definition of a continuum Gaussian free field via Gaussian Hilbert spaces. See \cite{SS07} for more details.

\begin{definition}
	We say that a family $\{\psi(f) : f \in \mathcal{H}\}$ of random variables on a probability space $(\Omega, \mathcal{F},\mathbb{P})$ is a Gaussian free field if the map $f \to \psi(f)$ is linear a.s. and each $\psi(f)$ is Gaussian with mean zero and variance
	\begin{equation}
	E(\psi(f)^2) = (f,-\Delta_\mathbb{R}^{-1}f).
	\end{equation}
	Furthermore, if $\Sigma = (q_{ij})$ is a positive semidefinite, non-degenerate matrix and $\Sigma = LL^T$ is its Cholesky decomposition, then $\tilde{\psi} = L\psi$ is called a Gaussian field with mean zero and covariance $Q^{-1}$ where $Q^{-1}$ is the inverse of the operator
	\begin{equation}
	Qf: = \sum^d_{i, j = 1}q_{ij}\frac{\partial^2}{\partial_i\partial_j}f.
	\end{equation}
\end{definition}

To state the result for the scaling limit, we also need the definition of linear functional on $\mathcal{H}$ induced by the gradient field $\phi$. Let $C^\infty_0(\mathbb{R}^d)$ denote the set of all infinitely differentiable functions $f :\mathbb{R}^d \mapsto \mathbb{R}$ with compact support. For any function $f \in C^\infty_0(\mathbb{R}^d)$ and the random field $\phi$ defined on the lattice $\mathbb{Z}^d$, we introduce the random linear functional $f \mapsto \phi(f)$ given by
\begin{equation}\label{defphioperator}
\phi(f) := \int dx ~f(x)\phi_{\lfloor x\rfloor},
\end{equation}
where $f$ satisfies
\begin{equation}\label{orthon}
\int dx ~f(x) = 0.
\end{equation}
Notice $f$ satisfies (\ref{orthon}) if  $f \in \mathcal{H}_0$. We will prove the following lemma which provides regularity estimates so that $\phi$ can be extended to a linear functional on $\mathcal{H}$. 
\begin{lemma}\label{reglemma}
	let $d \ge 3$ and $\mu$ be a translation-invariant, ergodic gradient Gibbs measure defined in (\ref{DLR}). There then exists a constant $c \in \infty$ such that for each $f \in \mathcal{H}_0$,
	\begin{equation}\label{reginequal}
	\norm{\phi(f)}_{L^2(\mu)} \le c\norm{f}_\mathcal{H}.
	\end{equation}
	In particular, $\phi$ is $L^2$ continuous and thus extends to a random linear functional $\phi : \mathcal{H} \to \mathbb{R}$.
\end{lemma}

For $\delta > 0$ and a function $f: \mathbb{R}^d \to \mathbb{R}$, let
\begin{equation}
f_\delta(x) := \delta^{d/2+1}f(\delta x).
\end{equation}
Note that if $f \in \mathcal{H}$ and $\delta < 1$, then
\begin{eqnarray*}
	\norm{f_\delta(x)}_\mathcal{H}^2& =& (f_\delta(x), f_\delta(x)) + (f_\delta(x), -\Delta_\mathbb{R}^{-1}f_\delta(x))\\
	&=&\delta^2(f, f) + (f, -\Delta_\mathbb{R}^{-1}f) \le \norm{f}_\mathcal{H}^2.
\end{eqnarray*}
Let $\phi_\delta$ be the linear functional on $C^\infty_0(\mathbb{R}^d)$ by
\begin{equation}
\phi_\delta(f) := \phi(f_\delta(x)).
\end{equation}
Then by Lemma \ref{reglemma}, $ \norm{\phi_\delta(f)}_{L^2(\mu)} = \norm{\phi(f_\delta)}_{L^2(\mu)} \le c\norm{f_\delta}_\mathcal{H} \le c\norm{f}_\mathcal{H}$, and thus $\phi_\delta$ can be extended to $\mathcal{H}$. Our main result is as follows.

\begin{thm}\label{main}
	Let $d \ge 3, \alpha < 1/2$, and $\mu$ be a translation-invariant, ergodic gradient Gibbs measure defined in (\ref{DLR}). Then for every $f \in \mathcal{H}$.
	\begin{equation}
	\lim_{\delta \to 0}\mathbb{E}_\mu(e^{i\phi_\delta(f)}) = \exp\left\{\frac{1}{2}\int dxf(x)(Q^{-1}f)(x)\right\},
	\end{equation}
	where $Q^{-1}$ is the inverse of the operator
	\begin{equation}\label{colmatrix}
	Qf: = \sum^d_{i, j = 1}q_{ij}\frac{\partial^2}{\partial_i\partial_j}f,
	\end{equation}
	with $(q_{ij})$ denoting some positive semidefinite, non-degenerate, $d\times d$ matrix depending only on $\alpha$ and $d$. In other words, the law of $\phi_\epsilon$ on the linear dual $\mathcal{E}'$ of any finite-dimensional linear subspace $\mathcal{E} \subset \mathcal{H}$ converges weakly to that of a Gaussian field with mean zero and covariance $(-Q)^{-1}$.
\end{thm}

\section{Moment Bounds}
\subsection{Bounds for the auxiliary field}
\hspace*{\parindent}We will start the proof of Theorem \ref{mymain1}, which gives the bounds on  $\left<e^{\lambda t_{jk}}\right>_{\alpha, \Lambda,\epsilon}$ for all $\lambda \in \mathbb{R}$. We will follow the strategy of the proof of Proposition 3 in \cite{BS12}. As all these bounds are uniform in $\epsilon$ and $\Lambda$, to shorten the notation, we write $\left<\cdot\right>_{\alpha, \Lambda,\epsilon}$ as $\left<\cdot\right>$. By scaling $\phi \to \sqrt{\beta}\phi$ and $\epsilon \to \epsilon/\beta^2$, we take $\beta = 1$ in the rest of the section. Recall that the symmetric finite difference operator $D_{\Lambda, \epsilon}$ is defined in (\ref{defD}). The following lemma about $D_{\Lambda, \epsilon}$ from \cite{BS12} is useful in the following proof.

\begin{lemma}[{{\cite[Lemma 1]{BS12}}}] \label{lemmalogcon}
	For all $t$, $\ln \det [D_{\Lambda,\epsilon}(t)]$ is a convex function of $t$. 
\end{lemma}

We start our proof by the uniform bound for the second moment of the $t$ field. The tool we use is the Brascamp-Lieb bound. The version of the Brascamp-Lieb bound is as follows.

\begin{thm}[{\cite[Brascamp-Lieb bound]{BL76}}]\label{BLBthm}
	Let $F$ be a convex function on $\mathbb{R}^n$, and let $A$ be a real, positive definite, $n \times n$ matrix. Assume $\exp[-(\phi, A\phi)-F(\phi)]  \in L^1$ and define 
	\[\left<k \right>_F = \frac{\int k(\phi)\exp[-(\phi, A\phi)-F(\phi)] d\phi}{\int\exp[-(\phi, A\phi)-F(\phi)] d\phi}.\]
	If $F(\phi)\equiv 0$ we write $\left<\cdot\right>_0$. Let $v \in \mathbb{R}^n$, $\alpha \ge 1$. Then
	\begin{equation}
	\left<|v \cdot \phi-\left<v\cdot\phi\right>_F|^\alpha\right>_F \le \left<|v\cdot \phi|^\alpha\right>_0
	\end{equation}
	when F is log concave. Furthermore, let $f_{xx}$ be the Hessian matrix of $F$ and $M$ be the covariance matrix
	\begin{equation}
	M_{ij} = \left<\phi_i\phi_j\right>_F-\left<\phi_i\right>_F\left<\phi_j\right>_F.
	\end{equation}
	Then
	\begin{equation}\label{blbvar}
	M\le \left<(2A+f_{xx})^{-1}\right>_F\le(2A)^{-1}.
	\end{equation} 
	
\end{thm}

\begin{lemma}
	There exists a constant $c(\alpha)$ such that
	\begin{equation}\label{tvarbound}
	\left<(t_{jk}-\left<t_{jk}\right>)^2\right> \le c(\alpha).
	\end{equation} 
	Furthermore, for all $\lambda \in \mathbb{R}$,
	\begin{equation}\label{firstbd}
	\left<e^{\lambda t_{jk}} \right> \le e^{c(\alpha)\lambda^2/2}e^{\lambda\left<t_{jk}\right>}.
	\end{equation} 
\end{lemma}
\begin{proof}
	For $\lambda \in \mathbb{R}$, let $F_\lambda = e^{\lambda t_{jk}}$, $q(\lambda) = \ln\left<F_\lambda\right>$. We bound $q(\lambda)$ using Taylor's theorem to second order in $\lambda$. To do this consider a $\lambda$ dependent measure $\left< \cdot \right>_\lambda: = \left<\cdot F_\lambda\right>/\left<F_\lambda\right>$. Then
	\begin{equation}\label{qdev}
	q(0) = 0, q'(0) = \left<t_{jk}\right>, q''(\lambda)= \left<(t_{jk}-\left<t_{jk}\right>_\lambda)^2\right>_\lambda.
	\end{equation}
	
	By the definition of $\left<\cdot\right>_\lambda$ and (\ref{modelgeneral}), the density of $t$ under $\left<\cdot\right>_\lambda$ is given by
	\begin{equation}
	\frac{1}{\left<F_\lambda\right>}\det[D_{\Lambda,\epsilon}(t)]^{-1/2}e^{\lambda t_{jk}}\prod_{il \in E(\Lambda)}\Big(\exp(-e^{t_{il}}+t_{il})f_\alpha(e^{t_{il}})dt_{il}\Big).
	\end{equation}  
	Define the {\itshape{action}} of $t$ variable by the negative of the logarithm of the unnormalized density of $t$. Then the action in the $t$ variables corresponding to $\left<\cdot\right>_\lambda$ is 
	\begin{equation}\label{deftactionlambda}
	EA_\lambda(t)=\frac{1}{2}\ln\det[D_\Lambda(t)]-\lambda t_{jk}+\sum_{il \in E(\Lambda)}\left(e^{t_{il}}-\ln f_\alpha(e^{t_{il}})- t_{il}\right), 
	\end{equation}
	which is convex by Lemma \ref{lemmalogcon} and Corollary \ref{flogconcave}. In fact, the Hessian of (\ref{deftactionlambda}) is
	bounded from below by the diagonal matrix $H = (h_{jk,il})_{jk,il \in E(\Lambda)}$ given by
	\begin{equation}
	\inf\left(e^{t}-\frac{d^2}{dt^2}\ln f_\alpha(e^{t})\right)\delta_{jk,il} = h(\alpha)\delta_{jk,il}.
	\end{equation}
	Furthermore,  $e^{t}-\frac{d^2}{dt^2}\ln f_\alpha(e^{t})$ goes to $+\infty$ both when $t\to-\infty$ by the second result of Proposition \ref{secondbound}, and when $t\to+\infty$ due to the term $e^{t}$. As a result, $h(\alpha) > 0$. In (\ref{blbvar}), let $A = H$ and $F = EA_\lambda(t)-\sum_{il\in E(\Lambda)}h(\alpha)t_{il}^2$, and then for all $\lambda \in \mathbb{R}$ we have
	\begin{equation}\label{qsecond}
	q''(\lambda)= \left<(t_{jk}-\left<t_{jk}\right>_\lambda)^2\right>_\lambda \le 	 h(\alpha)^{-1}/2.
	\end{equation}
	Let $\lambda = 0$ and $c(\alpha)=h(\alpha)^{-1}/2$, and then we have
	\begin{equation}
	\left<(t_{jk}-\left<t_{jk}\right>)^2\right> \le c(\alpha).
	\end{equation}
	
	By Taylor expansion with mean-value forms of the remainder, 
	\begin{equation}
	q(\lambda) = q(0)+q'(0)\lambda+q''(\xi_\lambda)\lambda^2/2.
	\end{equation}
	for some real number $\xi_\lambda \in [0, \lambda]$.
	Thus by (\ref{qdev}) and (\ref{qsecond}), we have
	\begin{equation}
	q(\lambda) \le \lambda\left<t_{jk}\right>+ c(\alpha)\lambda^2/2.
	\end{equation}
	Then by the definition of $q(\lambda)$,
	\begin{equation}
	\left<e^{\lambda t_{jk}} \right> \le e^{c(\alpha)\lambda^2/2}e^{\lambda\left<t_{jk}\right>}.
	\end{equation}	 
\end{proof}

Given (\ref{firstbd}), to bound $\left<e^{\lambda t_{jk}}\right>$, we still need bounds on $\left<t_{jk}\right>$. The proof of bounds on $\left<t_{jk}\right>$ is based on a Ward identity. The term Ward identity is used in theoretical physics to describe identities that arise by differentiating integrals with respect to a parameter that represents a continuous symmetry or approximate symmetry of the integrand. In our case the Ward identity is given by the following lemma.
\begin{lemma}
	Let $g(t) = \frac{d}{dt}\ln f_\alpha(e^{t})$. Then
	\begin{equation}\label{basiceqn}
	\left<g(t_{jk})\right>= \left<e^{t_{jk}}[1+\frac{1}{2}(\phi_j-\phi_k)^2]\right>-1.
	\end{equation}
\end{lemma}
\begin{proof}
	Recall that the partition function $Z_{\Lambda, \epsilon}(\alpha)$ is defined in (\ref{ameasure}) by
	\begin{equation}
	Z_{\Lambda,\epsilon}(\alpha)
	=  \int e^{-\sum_{jk}\left(1+(\phi_j-\phi_k)^2\right)e^{t_{jk}}-\sum_{j \in \Lambda}\epsilon \phi_j^2}\prod_{jk \in E}\Big(f_\alpha(e^{t_{jk}})e^{t_{jk}}dt_{jk}\Big)\prod_{j \in \Lambda}d\phi_j.
	\end{equation}
	Here we use we use a Ward identity generated by the change of variables
	\begin{equation}
	t_{jk}\to t_{jk}+b.
	\end{equation}
	Since the partition function does not depend on the constant $b$, the derivative with respect to $b$	evaluated at $b = 0$ vanishes hence
	\begin{equation}
	\left<-e^{t_{jk}}[1+\frac{1}{2}(\phi_j-\phi_k)^2]+\frac{d}{dt_{jk}}\ln f_\alpha(e^{t_{jk}})+1\right>=0.
	\end{equation}
	By the definition of $g(t)$, we have (\ref{basiceqn}).	
\end{proof}
We will derive both upper and lower bounds on $\left<t_{jk}\right>$ from (\ref{basiceqn}). The idea is to prove $\left<t_{jk}\right>$ satisfies inequalities with solution set bounded from above and below respectively. 
\begin{lemma}\label{tmeanupper}
	There exists a constant $C_u(\alpha)$ such that
	\begin{equation}
	\left<t_{jk}\right> \le C_u(\alpha).
	\end{equation}
\end{lemma}
\begin{proof}
	By the first result of Prop \ref{secondbound}, there exists constants $M$ and $C$, such that when $t < M$, $g(t) < C\exp(\frac{\alpha}{\alpha-1}t)+C$ and when $t \ge M$, $g(t) \le g(M)$. Then
	\begin{eqnarray}
	\left<g(t_{jk})\right> &=& \left<g(t_{jk}) \mathbbm{1}_{\{t_{jk}<M\}}\right>+\left<g(t_{jk}) \mathbbm{1}_{\{t_{jk} \ge M\}}\right>\nonumber\\
	&\le& \left<\left(C\exp(\alpha/(\alpha-1)t)+C\right)\nonumber \mathbbm{1}_{\{t_{jk}<M\}}\right>+\left<g(M) \mathbbm{1}_{\{t_{jk} \ge M\}}\right> \\
	&\le& \left<\left(C\exp(\alpha/(\alpha-1)t)\right)+C\right>+g(M).
	\end{eqnarray}
	By (\ref{firstbd}) with $\lambda = \alpha/(\alpha-1)$, we have
	\begin{equation}\label{rightbd}
	\left<g(t_{jk})\right> \le C+g(M)+ e^{c(\alpha)^{-1}(\alpha/(\alpha-1))^2/2}e^{(\alpha/(\alpha-1))\left<t_{jk}\right>}.
	\end{equation}
	
	On the other hand, by Jensen inequality
	\begin{equation}\label{leftbd}
	\left<e^{t_{jk}}[1+\frac{1}{2}(\phi_j-\phi_k)^2]\right> \ge \left<e^{t_{jk}}\right> \ge e^{\left<{t_{jk}}\right>}.
	\end{equation}
	Combining (\ref{basiceqn}), (\ref{rightbd}) and (\ref{leftbd}), we have
	\begin{equation}
	e^{\left<{t_{jk}}\right>} + 1 \le  C+g(M)+ e^{c(\alpha)^{-1}(\alpha/(\alpha-1))^2/2}e^{(\alpha/(\alpha-1))\left<t_{jk}\right>}.
	\end{equation}
	
	As $\alpha < \frac{1}{2}$, $\alpha/(\alpha-1) < 0$. Thus $C+g(M)+ e^{c(\alpha)^{-1}(\alpha/(\alpha-1))^2/2}e^{(\alpha/(\alpha-1))\left<t_{jk}\right>}$ is decreasing with respect to $\left<t_{jk}\right>$ while $	e^{\left<{t_{jk}}\right>} + 1$ is increasing. Therefore, for some positive constant $C_u(\alpha)$,
	\begin{equation}
	\left<t_{jk}\right> \le C_u(\alpha).
	\end{equation}
\end{proof}
Before moving to the lower bound, we first prove a formula regarding the operator $D_{\Lambda,\epsilon}$ defined in (\ref{defD}) and its inverse $G_{\Lambda,\epsilon}$.
\begin{lemma}For all $f \in \mathbb{R}^\Lambda$,
	\begin{equation}\label{inversesup}
	[f; G_{\Lambda,\epsilon}(t)f] = \sup\nolimits_\varphi (2[f ; \varphi] - [ \varphi; D_{\Lambda,\epsilon}(t)\varphi])
	\end{equation}
\end{lemma}
\begin{proof}
	By the definition of $D_{\Lambda,\epsilon}(t)$, $D_{\Lambda,\epsilon}(t)$ is a symmetric and positive definite invertible matrix. Thus so is its inverse $G_{\Lambda, \epsilon}(t)$.
	
	As $G_{\Lambda, \epsilon}(t)$ is positive definite, for all $f, \varphi \in \mathbb{R}^\Lambda$, we have
	\begin{equation}\label{positivedef}
	[f-D_{\Lambda, \epsilon}(t)\varphi, G_{\Lambda, \epsilon}(t)(f-D_{\Lambda, \epsilon}(t)\varphi)] \ge 0.
	\end{equation}
	By the linearity of the inner product, we have
	\begin{equation}
	[f, G_{\Lambda, \epsilon}(t)f]-[f,\varphi]-[D_{\Lambda, \epsilon}(t)\varphi,G_{\Lambda, \epsilon}(t)f]+[G_{\Lambda, \epsilon}(t)\varphi,\varphi] \ge 0.
	\end{equation}
	As $D_{\Lambda, \epsilon}(t)$ is a symmetric matrix, 
	\begin{equation}
	[f, G_{\Lambda, \epsilon}(t)f] \ge 2[f, \varphi]-[ \varphi; D_{\Lambda,\epsilon}(t)\varphi].
	\end{equation}
	As the above inequality holds for all $\varphi \in \mathbb{R}^\Lambda$, 
	\begin{equation}
	[f; G_{\Lambda,\epsilon}(t)f] \ge \sup\nolimits_\varphi (2[f ; \varphi] - [ \varphi; D_{\Lambda,\epsilon}(t)\varphi]).
	\end{equation}
	Furthermore, in (\ref{positivedef}), the equality holds when $f = D_{\Lambda, \epsilon}\phi$. Thus
	\begin{equation}
	[f; G_{\Lambda,\epsilon}(t)f] = \sup\nolimits_\varphi (2[f ; \varphi] - [ \varphi; D_{\Lambda,\epsilon}(t)\varphi]).
	\end{equation}
\end{proof}

\begin{lemma}\label{tmeanlower}
	There exists a negative constant $C_l(\alpha)$ such that
	\begin{equation}
	\left<t_{jk}\right> \ge C_l(\alpha).
	\end{equation}
\end{lemma}
\begin{proof}
	Referring to (\ref{defD}), by (\ref{inversesup}) and $[\phi;  D_{\Lambda, \epsilon}(t)\phi] \ge \beta e^{t_{jk}}(\phi_j-\phi_k)^2$ we have 
	\begin{eqnarray}
	(\phi_j - \phi_k)^2 &=& [(\delta_j-\delta_k); G_{\Lambda, \epsilon}(t)(\delta_j-\delta_k)] \nonumber\\
	&=&\sup\nolimits_\varphi (2[(\delta_j-\delta_k) ; \varphi] - [ \varphi; D_{\Lambda, \epsilon}(t)\varphi]) \nonumber\\
	&\le&\sup\nolimits_\varphi(2(\varphi_j-\varphi_k)-\beta e^{t_{jk}}(\phi_j-\phi_k)^2)  \nonumber\\
	&=&\sup\nolimits_\varphi(\frac{1}{\beta}e^{-t_{jk}}-\beta e^{t_{jk}}(\phi_j-\phi_k-\frac{1}{\beta}e^{-t_{jk}})^2) \nonumber\\
	& = & \frac{1}{\beta}e^{-t_{jk}}.
	\end{eqnarray}
	Thus if $g(t) = \frac{d}{dt}\ln f_\alpha(e^{t})$, then by (\ref{basiceqn})
	\begin{eqnarray}\label{secondbd}
	\left< g(t_{jk}) \right>&=&\left<e^{t_{jk}}\right>+\frac{1}{2}\left<e^{t_{jk}}(\phi_j-\phi_k)^2]\right>-1\nonumber\\ &\le& \left<e^{t_{jk}} \right>+\frac{1}{2\beta}-1.
	\end{eqnarray}
	By Corollary \ref{tailg}, $g(t_{jk})$ is decreasing and bounded below by some negative constant $K$. For a fixed constant $a > 0$, when $t <\left<t_{jk}\right>+a$, $g(t)\ge g(\left<t_{jk}\right>+a)$ by monotonicity, and  $t \ge\left<t_{jk}\right>+a$, $g(t)\ge K$ by the lower bound. Thus
	\begin{eqnarray}
	\left< g(t_{jk}) \right>&=&\left<g(t_{jk}) \mathbbm{1}_{\{t_{jk}<\left<t_{jk}\right>+a\}}\right>+\left<g(t_{jk}) \mathbbm{1}_{\{t_{jk} \ge \left<t_{jk}\right>+a\}}\right>\nonumber\\
	&\ge&\left<g(\left<t_{jk}\right>+a) \mathbbm{1}_{\{t_{jk}<\left<t_{jk}\right>+a\}}\right>+\left<K \mathbbm{1}_{\{t_{jk} \ge \left<t_{jk}\right>+a\}}\right>\nonumber\\
	&\ge&g(\left<t_{jk}\right>+a) \left< \mathbbm{1}_{\{t_{jk}<\left<t_{jk}\right>+a\}}\right>+K\nonumber\\
	&\ge&g(\left<t_{jk}\right>+a) \mathbb{P}\left({\{t_{jk}<\left<t_{jk}\right>+a\}}\right)+K.\label{glowerboud}
	\end{eqnarray}
	By Chebyshev's inequality,
	\begin{equation}
	\mathbb{P}\left({\{t_{jk} \ge \left<t_{jk}\right>+a\}}\right)\le \mathbb{P}\left({\{|t_{jk}- \left<t_{jk}\right>|\ge a\}}\right) \le a^{-2}\text{Var}(t_{jk}).
	\end{equation}
	By (\ref{tvarbound}), $\text{Var}(t_{jk}) \le c(\alpha)$, so $\mathbb{P}\left({\{t_{jk} \ge \left<t_{jk}\right>+a\}}\right) \le a^{-2}c(\alpha)$. Let $a = \sqrt{(c(\alpha))}/2$, and then we have 
	\begin{equation}\mathbb{P}\left({\{t_{jk}<\left<t_{jk}\right>+a\}}\right) >1- a^{-2}c(\alpha)=\frac{3}{4}.
	\end{equation}
	Combining this with (\ref{glowerboud}), we have
	\begin{equation}
	\left< g(t_{jk}) \right> \ge  \frac{3}{4}g\left(\left<t_{jk}\right>+\frac{\sqrt{(c(\alpha))}}{2}\right) +K.
	\end{equation}
	Combine this with (\ref{secondbd}), and by (\ref{firstbd}) with $\lambda = -1$ we have
	\begin{equation}
	\frac{3}{4}g\left(\left<t_{jk}\right>+\frac{\sqrt{(c(\alpha))}}{2}\right) +K \le e^{c(\alpha)/2}e^{\left<t_{jk}\right>}.
	\end{equation}
	By Corollary \ref{tailg}, $g(t) \to \infty$ when $t \to -\infty$. Therefore, for some negative constant $C_l(\alpha)$,
	\begin{equation}
	\left<t_{jk}\right> \ge C_l(\alpha).
	\end{equation}
\end{proof}

\begin{proof}[Proof of Proposition \ref{propbda}]
	For $\lambda < 0$, by (\ref{firstbd}) and Lemma \ref{tmeanlower}, we have
	\begin{equation}
	\left<e^{\lambda t_{jk}}\right> \le e^{c(\alpha)\lambda^2/2}e^{\lambda C_l(\alpha)}.
	\end{equation}
	For $\lambda > 0$, by (\ref{firstbd}) and Lemma \ref{tmeanupper}, we have
	\begin{equation}
	\left<e^{\lambda t_{jk}}\right> \le e^{c(\alpha)\lambda^2/2}e^{\lambda C_u(\alpha)}.
	\end{equation}
	Let $C(\lambda,\alpha) = \max\{e^{c(\alpha)\lambda^2/2}e^{\lambda C_u(\alpha)},e^{c(\alpha)\lambda^2/2}e^{\lambda C_l(\alpha)}\}$, and then we have
	\begin{equation}
	\left<e^{\lambda t_{jk}}\right> \le C(\lambda, \alpha)
	\end{equation}
\end{proof}

\subsection{Bounds of the Moments}
\hspace*{\parindent}Given the joint distribution (\ref{defextmeasure}), conditional on the $t$ field, the distribution of the $\phi$ field is a massive Gaussian free field with covariance $G_{\Lambda, \epsilon}(t)$ which is the inverse of $D_{\Lambda,\epsilon}$ defined in (\ref{defD}). 
Thus the finite moments of $\phi$ are represented by the expectation over the $t$ field
\begin{equation}\label{greenbd}
\left<(\phi\cdot v)^{2n}\right>_{\alpha, \Lambda,\epsilon} = \left<(2n-1)!!([v;G_{\Lambda,\epsilon}(t)v]^{n})\right>_{\alpha, \Lambda,\epsilon}.
\end{equation} 

Recall that the lattice Laplacian $-\Delta^p_\Lambda$ with periodic boundary condition and its inverse $G^p_{\Lambda}$ on $D_p$ are defined in (\ref{deflaplace}). 
for all $\phi : \Lambda \to \mathbb{R}$ satisfying periodic boundary condition. Lemma 2 in \cite{BS12} gives the following bound of $[v;G_{\Lambda,\epsilon}(t)v]$ by $G^p_{\Lambda}$. In preparation for the following proof of the existence infinite volume Gibbs measure, here we extend this Lemma with the case $\Lambda = \mathbb{Z}^d$. To state the theorem, let $\mathcal{D}_0$ be the subset of $\mathbb{R}^{\mathbb{Z}^d}$ representing the collection of all functions on $\mathbb{Z}^d$ with compact support. For $v \in \mathcal{D}_0$, define the lattice Laplacian by
\begin{equation}\label{defgradientsquare}
[v, -\Delta_d v] = \sum_{xy \in E}(v_y-v_x)^2.
\end{equation} Then $G_d(x, y)$ is said to be a lattice Green function if $G_d(x, y)$ is symmetric in $x$ and $y$ and $G$ is a solution to
\begin{equation}\label{defgreenl}
(-\Delta_d)G_d(x,y) = \delta(x-y)
\end{equation}
with $\delta(x - y)$ being the identity matrix. It is well-known that $G_d(x, y)$ exist if and only if $d \ge 3$ (see section 3.2 of \cite{Fun03}). Similar to (\ref{defD}), define the symmetric difference operator $D_{d, \epsilon}(t)$ by the quadratic form
\begin{equation}\label{defDd}
[f; D_{d,\epsilon}(t)f] = \sum_{jk\in E}(f_j - f_k)^2e^{t_{jk}}+\epsilon\sum_{j \in \mathbb{Z}^d} f_j^2
\end{equation}
for all $f \in \mathcal{D}_0$. Let $G_{d,\epsilon}(t) = (D_{d,\epsilon}(t))^{-1}$ be the Green's function.
\begin{lemma}[{{\cite[Lemma 2]{BS12}}}]\label{lemmabd}
	\begin{enumerate}
		\item If $v \in \mathcal{D}_p$, then
		the Green's function $G_{\Lambda,\epsilon}(t)$ satisfies the quadratic form bound
		\begin{equation}\label{lemmagreenbd}
		0 \le [v, G_{\Lambda,\epsilon}(t)v] \le \sum_{jk \in E(\Lambda)}((G^p_{\Lambda}v)_j-(G^p_{\Lambda}v)_k)^2e^{-t_{jk}};
		\end{equation}
		\item For $d \ge 3$, if $v \in \mathcal{D}_0$, then
		the Green's function $G_{d,\epsilon}(t)$ satisfies the quadratic form bound
		\begin{equation}\label{greenboundinf}
		0 \le [v, G_{d,\epsilon}(t)v] \le \sum_{jk \in E}((G_dv)_j-(G_dv)_k)^2e^{-t_{jk}}.
		\end{equation}
	\end{enumerate}
\end{lemma}
\begin{proof}
	The first part is Lemma 2 in \cite{BS12}. The proof of the second part is the same as that of the finite volume case in the first part.
\end{proof}

Then the proof of Theorem \ref{mymain1} is a combination of Proposition \ref{propbda}, (\ref{greenbd}) and Lemma \ref{lemmabd}.

\begin{proof}[Proof of Theorem \ref{mymain1}]
	By Lemma \ref{lemmabd}, 
	\begin{eqnarray}
	[v;G_{\Lambda, \epsilon}(t)v]^{2n}& \le&\left(\sum_{ij \in E(\Lambda)}\left((G^p_{\Lambda}v)_i-(G^p_{\Lambda}v)_j\right)^2e^{-t_{ij}} \right)^{n}\nonumber\\
	&=& \sum_{i_1j_1, \cdots i_{n}j_{n}}\prod_{k = 1}^{n}\left((G^p_{\Lambda}v)_{i_k}-(G^p_{\Lambda}v)_{j_k}\right)^2e^{-t_{i_kj_k}}.\label{expansion}
	\end{eqnarray}
	
	By the H\"older inequality
	\begin{equation}
	\left<\exp(\sum -t_{i_kj_k})\right>_{\alpha, \Lambda, \epsilon} \le \prod_{k=1}^{n} \left<\exp(-nt_{i_kj_k})\right>_{\alpha, \Lambda, \epsilon}^{1/n}.
	\end{equation}
	
	By Proposition \ref{propbda}, $\left<\exp(-nt_{i_kj_k})\right>_{\alpha, \Lambda, \epsilon}$ is bounded from above by $C(n, \alpha)$. Thus
	\begin{equation}
	\left<\prod_{k=1}^n\exp( -t_{i_kj_k})\right>_{\alpha, \Lambda, \epsilon} \le \prod_{k=1}^{n} \left<\exp(-nt_{i_kj_k})\right>_{\alpha, \Lambda, \epsilon}^{1/n} \le C(n, \alpha).
	\end{equation}
	Combing this bound with  (\ref{expansion}), we have
	\begin{eqnarray}
	\left<([v;G(t)v]^{2n})\right>_{\alpha, \Lambda, \epsilon} &\le&\left<\sum_{i_1j_1, \cdots i_{n}j_{n}}\prod_{k = 1}^{n}\left((G^p_{\Lambda}v)_{i_k}-(G^p_{\Lambda}v)_{j_k}\right)^2e^{-t_{i_kj_k}}\right>_{\alpha, \Lambda, \epsilon} \nonumber\\
	&= &\sum_{i_1j_1, \cdots i_{n}j_{n}}\prod_{k = 1}^{n}\left((G^p_{\Lambda}v)_{i_k}-(G^p_{\Lambda}v)_{j_k}\right)^2\left<e^{-\sum t_{i_kj_k}}\right>_{\alpha, \Lambda, \epsilon}\nonumber\\
	&\le &  C(n, \alpha)\sum_{i_1j_1, \cdots i_{n}j_{n} }\prod_{k = 1}^{n}\left((G^p_{\Lambda}v)_{i_k}-(G^p_{\Lambda}v)_{j_k}\right)^2\nonumber\\
	&=& C(n, \alpha)\left(\sum_{ij \in E(\Lambda)}\left((G^p_{\Lambda}v)_i-(G^p_{\Lambda}v)_j\right)^2 \right)^{n}\nonumber\\
	&=&C(n,\alpha)[G^p_{\Lambda}v,(-\Delta^p_\Lambda)G^p_\Lambda v]^n\nonumber\\
	&=&C(n,\alpha)[v,G^p_\Lambda v]^n.
	\end{eqnarray}
	Notice that in the last equality we use the fact that $G^p_\Lambda = (-\Delta^p_\Lambda)^{-1}$.
	
	Let $\tilde{C}(n, \alpha) = (2n-1)!!C(n, \alpha)$. Thus by (\ref{greenbd}),
	\begin{equation}
	\left<(\phi\cdot v)^{2n}\right>_{\alpha, \Lambda, \epsilon} \le \tilde{C}(n, \alpha) [v, G^p_\Lambda v]^n.\label{mombound}
	\end{equation}
	This finishes the proof of Theorem \ref{mymain1}.
\end{proof}

\subsection{Existence of infinite volume measure}
\hspace*{\parindent}Now we give the proof of Theorem \ref{existGibbs}. To show the existence of the infinite volume massless Gibbs measure, the basic idea is to show the existence of the weak limit of a sequence of massive infinite volume measures. Recall that the definition of infinite volume massless Gibbs measure is given by the DLR-equation in Definition \ref{defDLR} and $\mathscr{G}$ denotes the set of all infinite volume Gibbs measures.
 
Recall that the massive Hamiltonian $H_{\Lambda, \epsilon}$ is defined in (\ref{defHm}). 
Similar to massless case, given a finite set $\Lambda \subset \mathbb{Z}^d$ and $\psi \in \Omega$, we define the finite volume massive Gibbs measure  over $\Lambda$ given by
\begin{equation}\label{modelamass}
\mu^{\psi}_{\Lambda, \epsilon} := \frac{1}{Z_{\Lambda, \epsilon}(\psi_{\Lambda^c})}e^{-H_{\Lambda, \epsilon}(\phi\vee\psi)}\prod_{j \in \Lambda}d\phi_j,
\end{equation}
where $Z_{\Lambda,\beta,\epsilon}(\psi_{\Lambda^c})$ is the normalization constant. Then the massive Gibbs measure $\mu_\epsilon$ (on $\mathbb{Z}^d$) is defined by means of the DLR equation with the local specifications $\mu_{\Lambda, \beta, \epsilon}$ in place of $\mu_{\Lambda, \beta}$ in Definition (\ref{defDLR}). For every fixed $\epsilon$, let $\mathscr{G}_\epsilon$ denote the set of all massive infinite volume Gibbs measures. 

Recall that the massive finite volume measure with periodic boundary condition is defined in (\ref{model}). Let $\mathscr{G}^p_\epsilon$ be the set of all cluster points of $\{\mu^p_{\Lambda,\epsilon}|{\Lambda=\mathbb{Z}^d_N}, N\ge 3\}$. The following lemma states some basic properties of  $\mathscr{G}^p_\epsilon$.

\begin{lemma}\label{lemmainfitevolumemeasure1}
	~
	\begin{itemize}
		\item \cite[Theorem 18.12]{Georg11} For each $\epsilon > 0$, $\mathscr{G}^p_\epsilon$ is non-empty;
		\item \cite[Example 5.20.3]{Georg11} $\mathscr{G}^p_\epsilon \subset \mathscr{G}_\epsilon$; Furthermore, any measure in $\mathscr{G}^p_\epsilon$ is translation invariant.	
	\end{itemize} 
\end{lemma}

In this section, we will firstly introduce the definition of Green functions on $\mathbb{Z}^d$ and give a quadratic form bound of the Green function for $\mathbb{Z}^d$ case. Then we will give the proof of Theorem \ref{existGibbs} based on this bound.

Now we introduce the extended infinite volume measure on $\mathbb{R}^{\mathbb{Z}^d}\times\mathbb{R}^E$. Recall that the finite volume extended measure $\hat{\mu}^p_{\Lambda, \epsilon}$ is given in (\ref{defextmeasureperiodic}). For fixed $\epsilon$, let $\widehat{\mathscr{G}}^p_\epsilon$ be the set of cluster points of $\hat{\mu}_{\Lambda, \epsilon}$.  By the Remark 2.1 of \cite{BK07}, there is a one-to-one correspondence between the infinite volume measure on $\phi$'s in ${\mathscr{G}}^p_\epsilon$ and the infinite volume measure on $(\phi, t)$'s in $\hat{\mathscr{G}}^p_\epsilon$. Explicitly, by (2.7) in  \cite{BK07}, if $\mu_\epsilon \in {\mathscr{G}}^p_\epsilon$, then the corresponding $\hat{\mu}_\epsilon \in \hat{\mathscr{G}}^p_\epsilon$ is defined by (extending the consistent family of measures of the form)
\begin{eqnarray}
&&\hat{\mu}_\epsilon((\phi_x, t_{jk})_{x \in \Lambda, jk \in E(\Lambda)} \in A \times B)\nonumber\\ &:= &\int_B \prod_{jk \in E(\Lambda)}f_\alpha(e^{t_{jk}})e^{t_{jk}}dt_{jk}\nonumber\\&&\quad\quad\mathbb{E}_{\mu_\epsilon}\left(\mathbbm{1}_A\prod_{jk \in E(\Lambda)}e^{V(\phi_j-\phi_k)-(\phi_j-\phi_k)^2e^{t_{jk}}}\right)\label{defonemapping},
\end{eqnarray}
where $\Lambda$ is a torus in $\mathbb{Z}^d$.
Notice that $V$ in the exponent cancels part of the interaction in the infinite volume measure $\hat{\mu}_\epsilon$ and then the integral over $t$ in $B$ will restore it by (\ref{lpstable}). On the other hand, the $\phi$-marginal of $\hat{\mu}_\epsilon$ gives us back $\mu_\epsilon$.
Furthermore, by direct inspection of (\ref{defonemapping}), conditional on $t$, the conditional distribution of $\phi$ is a multivariate Gaussian law. Then by the property of the multivariate Gaussian law, for $v \in \mathcal{D}_0$ we have
\begin{equation}\label{gaussioncov}
\left<(\phi\cdot v)^{2n}\right>_{\mu_\epsilon} = \left<(2n-1)!!([v;G_{d,\epsilon}(t)v]^{n})\right>_{\tilde{\mu}_\epsilon}
\end{equation}

\begin{proof}[Proof of Theorem \ref{existGibbs}]
Firstly, for each fixed $\epsilon$, recall the $\mathscr{G}_\epsilon$ is the set of all cluster points of $\mu^p_{\Lambda, \beta, \epsilon}$. By Lemma \ref{lemmainfitevolumemeasure1}, $\mathscr{G}^p_\epsilon$ is non-empty and $\mathscr{G}^p_\epsilon \subset \mathscr{G}_\epsilon$.
Consider a sequence $\{\mu_n | \mu_n \in \mathscr{G}^p_{1/n} \}$. Then $\mu_n$ is an infinite volume Gibbs measure with respect to finite Gibbs measures (\ref{model}) with $\epsilon = 1/n$.  
Let $\Lambda$ be a finite subset of $\mathbb{Z}^{d}$. Given $\mathscr{F}_\Lambda$-measurable bounded function $f$ and $\mathscr{F}_{\Lambda^c}$-measurable bounded local function $g$, we have that by DLR equation
	\begin{equation}\label{DLRmass}
	\mathbb{E}_{\mu_n}(fg) =  \mathbb{E}_{\mu_n}\left(g(\psi)\mathbb{E}_{\mu_{\Lambda, 1/n}^\psi}(f)\right)
	\end{equation}
where $\mu_{\Lambda, 1/n}^\psi$ is the massive finite volume Gibbs measure with the boundary condition $\psi$ defined in (\ref{modelamass}).
Furthermore, each $\mu_n$ is translation invariant.
Now we will prove that  $\mu_n$'s converge to a massless Gibbs measure as $n$ goes infinity.
	
For each $n$, let $\hat{\mu}_n$ be the extended Gibbs measure with respect to $\mu_n$ defined in (\ref{defonemapping}). Let $\mathbb{P}_n$ and $\left<\cdot\right>_n$ be the probability and expectation with respect to $\mu_n$ respectively, and $\left<\cdot\right>_{\hat{n}}$ be the expectation with respect to $\hat{\mu}_n$. We will prove that there is a probability measure $\nu$ on $\{\Omega, \mathscr{F}_{\mathbb{Z}^d}\}$ which is a subsequence limit of $\mu_n$ in weak sense.
	
To show this, it suffices to show the tightness of  $\mu_n$. Introduce weighted $\ell^2$ norm on $\Omega$ by
	\begin{equation}\label{defl2norm}
	\norm{\phi}^2_r = \sum_{x \in \mathbb{Z}^d} \phi(x)^2e^{-2r|x|}
	\end{equation}
for $r > 0$. By Proposition 3.3 of \cite{Fun03}, for $M > 0$, $K_M = \{\phi \in \Omega | \norm{\phi}_r \le M \}$ is a compact set in $\Omega$.  Then
	\begin{eqnarray}
	\mathbb{P}_n(K_M^c)&=&\left<\mathbbm{1}\left\{\norm{\phi}_r > M\right\}\right>_n\nonumber\\
	&\le&\left<\norm{\phi}^2_r/M^2\right>_n\nonumber\\
	&=&\sum_{x \in \mathbb{Z}^d}\left<\phi(x)^2\right>_ne^{-2|x|}/M^2.\label{cheb}
	\end{eqnarray} 
By translation invariance of $\mu_n$, $\left<\phi(x)^2\right>_n = \left<\phi(0)^2\right>_n$ for all $x \in \mathbb{Z}^d$. Let $v = \delta_0$ and then by  (\ref{greenboundinf}) and (\ref{gaussioncov}),  we have
	\begin{equation}
	\left<\phi_0^2\right>_n = \left<[v,G_{d, \epsilon}(t)v]\right>_{\hat{n}} \le \sum_{jk \in E}((G_dv)_j-(G_dv)_k)^2\left<e^{-t_{jk}}\right>_{\hat{n}}.
	\end{equation}
By (\ref{tbd}), $\left<e^{-t_{jk}}\right>_{\hat{n}}$ is bounded above by constant $C$. As $G_d^{-1} = -\Delta_d = \nabla^*\nabla$, we have
	\begin{equation}\label{phitightness}
	\left<\phi_0^2\right>_n \le C\sum_{jk \in E}((G_dv)_j-(G_dv)_k)^2 = C[v, G_d v] = CG_d(0,0).
	\end{equation}
As discussed in last section, $G_d(0, 0) < \infty$ if and only if $d \ge 3$. Combine this with (\ref{cheb}), and we have
	\begin{equation}
	\mathbb{P}_n(K_M^c) \le \left(C\sum_{x \in \mathbb{Z}^d}e^{-2|x|}/M^2\right)G_d(0, 0).
	\end{equation}
As the right-hand side of the inequality is independent with $n$, this gives the tightness of $\{\mu_n\}$. Thus there exists a probability measure $\nu$ on $\{\Omega, \mathscr{F}_{\mathbb{Z}^d}\}$ which is a subsequence limit of $\mu_n$ in weak sense. 
	
Now we will prove that $\nu$ is a translation invariant infinite volume Gibbs measure. When $n$ goes to infinity,
$\mathbb{E}_{\mu_{\Lambda, 1/n}^\psi}(f)$ converges to $\mathbb{E}_{\mu_{\Lambda}^\psi}(f)$ where $\mu_{\Lambda}^\psi$ is defined in (\ref{modela}). As $f$ and $g$ are local and bounded, the left and the right hand of (\ref{DLRmass}) converges to $\mathbb{E}_{\nu}(fg)$ and $\mathbb{E}_{\nu}\left(g(\psi)\mathbb{E}_{\mu_{\Lambda}^\psi}(f)\right)$. Thus $\nu$ satisfies DLR equation (\ref{DLR}). As $\mu_n$'s is translation invariant, thus so is $\nu$.
	
Let $\mathscr{G}_\Theta \subset \mathscr{G}$ be the set of all translation invariant infinite volume Gibbs measures. Then $\mathscr{G}_\Theta$ is non-empty as $\nu \in \mathscr{G}_\Theta$. By Theorem 14.
15 in \cite{Georg11}, there exists an ergodic Gibbs measure $\mu \in \mathscr{G}$ which is the extreme point of $\mathscr{G}$ as a convex set. Then $\mu$ satisfies all the requirements in Theorem \ref{existGibbs}.
\end{proof}

\section{Scaling limit}
\subsection{Connection to random conductance model}
\hspace*{\parindent} It is well known that the mean and covariance matrix of the Gaussian system, both massless one and massive one, have random walk representations in terms of the continuous time simple random walk in $\mathbb{Z}^d$ (Proposition 3.2 and 3.8 in \cite{FT2005}). In fact, similar results also hold for our model where the corresponding stochastic process is random walk in random environment.  Here we give the definition of random walk in random environment. Previous results in random walk in random environment required in the following proof are listed in appendix \ref{sec::rwre}. Let $d$ be the natural graph distance on $\mathbb{Z}^d$, i.e. $d(x, y)$ is the minimal length of a path between $x$ and $y$. We denote by $B(x, r)$ the closed ball with center $x \in \mathbb{Z}^d$ and radius $r$, i.e. $B(x, r) := \{y \in \mathbb{Z}^d | d(x, y) \le r\}$.

A positive weight $\omega$ is a map from $E \to (0, \infty)$. This induces a conductance matrix which we also call $\omega$, that is for $x, y \in \mathbb{Z}^d$ we set $\omega(x, y) = \omega(y, x) = \omega(\{x, y\})$ if $\{x, y\} \in E$ and $\omega(x, y) = 0$ otherwise. Let us further
define measures $u^\omega$ and $v^\omega$ on $\mathbb{Z}^d$ by 
\begin{equation}\label{defuv}
u^\omega(x) := \sum_{y \sim x}\omega(x, y) \text{~~~~and~~~~} 
v^\omega(x) := \sum_{y \sim x}\frac{1}{\omega(x, y)}.
\end{equation}
Let $\Omega_\omega = (0, \infty)^E$ be the set of all weights. We will henceforth denote by $\mathbb{P}$ a probability measure on $(\Omega_\omega, \mathscr{F}_\omega) = ((0, \infty)^E, \mathscr{B}((0, \infty))^{\otimes E})$, 
and we write $\mathbb{E}$ to
denote the expectation with respect to $\mathbb{P}$.

A space shift by $z \in \mathbb{Z}^d$ is the map $\tau_z : \Omega_\omega \to \Omega_\omega$
\begin{equation}
(\tau_z\omega)(x, y) := \omega(x+z,y+z),\quad \forall \{x, y\} \in E.
\end{equation}
The set $\{\tau_x : x \in \mathbb{Z}^d\}$ together with the operation $\tau_x \circ \tau_y := \tau_{x+y}$ defines the group of space shifts.

For any fixed $\omega$ we consider a reversible continuous time Markov chain, $X = \{X_t: t \ge 0\}$, on $\mathbb{Z}^d$ with generator $\mathcal{L}^\omega_X$
acting on bounded functions $f : \mathbb{Z}^d \to \mathbb{R}$ defined by
\begin{equation}\label{generator}
(\mathcal{L}^\omega_Xf)(x) = \sum_{y \sim x}\omega(x, y)(f(y)-f(x)).
\end{equation}
This process is called the variable speed random walk (VSRW), and has an exponential time  at $x$ with mean $1/u^\omega(x)$. We denote by $P^\omega_x$ the law of the process starting at the vertex $x \in \mathbb{Z}^d$. The corresponding expectation will be denoted by $E_x^{\omega}$. For $x, y \in B$ and $t \ge 0$ let $p^\omega (t, x, y)$ be the transition densities of $X$ with respect to the counting measure (which is the reversible measure of $X$), which are also known as the heat kernels associated with $\mathcal{L}^\omega_X$, i.e.
\begin{equation}\label{heatkernel}
p^\omega (t, x, y) := P^{\omega}_x[X_t = y].
\end{equation}  

Recall that $\Omega = \mathbb{R}^{\mathbb{Z}^d}$ is the space of fields on the vertices, while $\Omega_\omega = (0, \infty)^E$ is the set of all weights on edges. To state the random walk representation for our model, we first extend the infinite volume measure to the space $\Omega\times\Omega_\omega$. Given an infinite volume Gibbs measure $\mu$ defined in (\ref{DLR}) let $\rho(\omega): \mathbb{R}^+ \mapsto \mathbb{R}^+$ be the probability density that satisfies 
\begin{equation}\label{weightdensity}
\exp\left({-\big(1+\beta x^2\big)^\alpha}\right) = \int_0^\infty \exp{\left(-\frac{1}{2}\omega x^2\right)}\rho(\omega)d\omega.
\end{equation}
For each finite $M \subset E$, consider the measure $\tilde{\mu}_M$ on $\Omega \times (\mathbb{R}^+)^M$ defined by
\begin{equation}\label{muext}
\tilde{\mu}_M(A \times B) := \int_B \prod_{e \in M}\rho(\omega_e)d\omega_e\mathbb{E}_\mu\left(\mathbbm{1}_A\prod_{jk \in M}e^{V(\phi_j-\phi_k)-\frac{1}{2}\omega_{jk}(\phi_j-\phi_k)^2}\right),
\end{equation}
where $A \in \mathscr{F}_{\mathbb{Z}^d}$ and $B \in \mathscr{B}((\mathbb{R}^+)^M)$ are Borel sets and $V(\phi_j-\phi_k) = \left(1+\beta(\phi_j-\phi_k)^2\right)^\alpha$.
\begin{remark}
	The extension in (\ref{muext}) is equivalent with that given in (\ref{defonemapping}), but we restate it in the notation of the random conductance model.
\end{remark}

(\ref{weightdensity}) ensures
that $\tilde{\mu}_M$ is a consistent family of measures; by Kolmogorov's extension
theorem, these are projections from a unique measure $\tilde{\mu}$ onto configurations $(\phi,\omega) \in \Omega
\times \Omega_\omega$. The $\phi$ marginal of $\tilde{\mu}$ is $\mu$. Using the name in \cite{BS11}, we call $\tilde{\mu}$ {\itshape{extended gradient Gibbs measure}} as it is in fact a Gibbs measure with Hamiltonian $\frac{1}{2}\omega_{jk}(\phi_j-\phi_k)^2$. 
The following lemmas from \cite{BS11} characterize the properties of $\tilde{\mu}$.
\begin{lemma}[{\cite[Lemma 3.2]{BS11}}]\label{lemmamuext}
	Let $\mu$ be a gradient Gibbs measure and let $\tilde{\mu}$ be its extension
	to $\Omega \times \Omega_\omega$. If $\mu$ is translation-invariant and ergodic, then so is $\tilde{\mu}$.
\end{lemma}

\begin{remark}
	With a simple change of variables, we see that if $\rho$ satisfies (\ref{weightdensity}) and $f_\alpha$ is given by (\ref{lpstable}), then
	\begin{equation}
	\rho(\omega) = \frac{1}{2\beta} f_\alpha(\frac{\omega}{2\beta})e^{-\frac{\omega}{2\beta}},
	\end{equation}
	and that $e^{t_{jk}}$ has the same distribution with $\omega_{jk}/2\beta$ in (\ref{muext}). We introduce the auxiliary field $t_{jk}$ in previous sections to make the effective action convex by lemma \ref{lemmalogcon} and properties of $\alpha$-stable density. Here we introduced a new auxiliary field $\omega_{jk}$ to connect our model to random walk in random environment.
	
\end{remark}

Combining this remark with Proposition \ref{propbda}, we have following corollary.

\begin{col}\label{pbound}
	For any $p \in \mathbb{R}$ and $e \in E$, $\mathbb{E}_{\tilde{\mu}}(\omega^p(e))$ exists.
\end{col}

The next proposition shows that conditional on the weight configuration, the extended gradient Gibbs measure has a random walk representation with respect to random walk with weight $\omega$ defined in (\ref{generator}). To state the proposition, define the $\sigma$-field $\mathscr{E} := \sigma(\{\omega_b : b \in E \})$.
\begin{prop}\label{Propconditionexp}
	For $d \ge 3$, let $\mu$ be a translation-invariant, ergodic gradient Gibbs measure defined in (\ref{DLR}) and $\tilde{\mu}$ is the extension of $\mu$ to $\Omega \times \Omega_\omega$. For $\tilde\mu$-a.e. $\omega$, the conditional law $\tilde{\mu}( \cdot|\mathscr{E})(\omega)$, regarded as a measure on $\Omega$, is Gaussian with constant mean
	and covariance given by $(−\mathcal{L}^\omega_X)^{-1}$. Explicitly, for each $f :\mathbb{Z}^d \to \mathbb{R}$ with finite support and $\sum_x f(x) = 0$,
	\begin{eqnarray}
	\mathbb{E}_{\tilde{\mu}}\left(\left.\sum_xf(x)\phi_x\right|\mathscr{E}\right)(\omega) &=& 0;\\
	\text{Var}_{\tilde{\mu}}\left(\left.\sum_xf(x)\phi_x\right|\mathscr{E}\right)(\omega) &=& \sum_xf(x)\left((−\mathcal{L}^\omega_X)^{-1}f\right)(x).
	\end{eqnarray}
\end{prop}

To prove Proposition \ref{Propconditionexp}, we first introduce the shift covariance property.   For a function $g: \Omega_\omega\times\mathbb{Z}^d \to \mathbb{R}$, we say $g$ satisfies {\itshape shift covariance property} if
\begin{equation}\label{shiftcov1}
g(\omega, x+b) - g(\omega, x) = g(\tau_x\omega, b),
\end{equation}
where  $x \in \mathbb{Z}^d$ and $b$ is a coordinate unit vector in $\mathbb{Z}^d$, and 
\begin{equation}\label{shiftcov2}
g(\omega, 0) = 0;
\end{equation}
we say $g$ is harmonic if
\begin{equation}\label{harm}
\mathcal{L}^\omega_Xg(\omega, \cdot) = 0, ~~~\mathbb{P}-a.s.~ \omega.
\end{equation}
The next lemma, which is similar to Lemma 3.3 in \cite{BS11}, shows that  harmonic, shift-covariant functions are uniquely determined by their expectation with respect to ergodic measures on the conductances.

\begin{lemma}\label{lemmazerofunction}
	Let $\nu$ be a translation-invariant, ergodic probability measure on configurations $\omega = (\omega_b) \in \Omega_\omega$. Let $g: \Omega_\omega^E\times\mathbb{Z}^d \to \mathbb{R}$ be a measurable function which is:
	\begin{enumerate}
		\item harmonic in the sense of (\ref{harm}), $\nu$-a.s.;
		\item shift-covariant in the sense of (\ref{shiftcov1}) and (\ref{shiftcov2}), $\nu$-a.s.;
		\item square integrable for each component in the sense that $\mathbb{E}_{\tilde{\nu}} |g(\omega,x)|^2 < \infty$ for all $x$ with $|x| = 1$;
		\item square integrable as a vector field $\mathbb{E}_{\tilde{\nu}} \left(\sum_{|x| = 1}\omega_{0x}|g(\omega,x)|^2\right) < \infty$.
	\end{enumerate}
	If $\mathbb{E}_\nu(g(\cdot,x)) = 0$ for all $x$ with $|x| = 1$, then $g(\cdot,x) = 0$ a.s. for all $x \in \mathbb{Z}^d$.
\end{lemma}

Compared with \cite[Lemma 3.3]{BS11}, we make two changes in the assumption: remove the condition of uniformly elliptic for conductance, namely $\mu(\epsilon \le \omega_b \le 1/\epsilon) = 1$ for some $\epsilon > 0$; add condition 4. In fact, in the proof of \cite[Lemma 3.3]{BS11}, the condition of  uniformly elliptic for conductance is used only to prove a equivalence of condition 3 and 4.  We defer the proof, and further discussion of the consequences of shift covariance and harmonicity, to Appendix \ref{sec::potential}.

Now we will give the proof of Proposition \ref{Propconditionexp}. The proof here follows the strategy of the proof of Lemma 3.4 in \cite{BS11}.
\begin{proof}[Proof of Proposition \ref{Propconditionexp}]
	The fact that the conditional measure is a multivariate Gaussian law with covariance $(\mathcal{L}^\omega_X)^{-1}$
	can be checked by direct inspection of (\ref{muext}).
	The only nontrivial task is to identify the mean. Define $u: \Omega_\omega\times\mathbb{Z}^d \to \mathbb{R}$ by
	\begin{equation}
	u(\omega, x) = 	\mathbb{E}_{\tilde{\mu}}(\phi_x-\phi_0|\mathscr{E})(\omega). 
	\end{equation}
	We claim that $u$ satisfies all conditions in Lemma \ref{lemmazerofunction}. First, to prove $u$ is harmonic in the sense of (\ref{harm}), consider
	\begin{equation}
	\mathcal{L}^\omega_Xu(\omega, x) = \mathbb{E}_{\tilde{\mu}}\left(\left.\sum\limits_{y: |y-x|=1}\omega_{xy}(\phi_y-\phi_x)\right|\mathscr{E}\right)(\omega).
	\end{equation}
	Using the fact that the $\tilde{\mu}$ is Gibbs, conditional on $\sigma(\phi_y;y\not=x)$ the conditional measure $\mu_{\left\{x\right\}}$ is Gaussian with the explicit form
	\begin{equation}
	\mu_{\left\{x\right\}}(d\phi_x)=\frac{1}{Z}\exp \left\{-\frac{1}{2} \phi_x^2 \sum\limits_{y: |y-x|=1}\omega_{xy}+\phi_x \sum\limits_{y: |y-x|=1}\omega_{xy}\phi_y\right\}
	\end{equation}
	where $Z$ is an appropriate normalization constant. By change of variables $\phi_x \to \phi_x + a$ and differentiating with respect to $a$ at $a = 0$, we have the mean of $\phi_x\sum\limits_{y: |y-x|=1}w_{xy}$ under $\mu_{\left\{x\right\}}$ is exactly
	$\sum\limits_{y: |y-x|=1}w_{xy}\phi_y$, proving that
	$\mathcal{L}^\omega_Xu(\omega, x) = 0 $.
	
	Next, we observe that the translation invariance of $\tilde{\mu}$ implies that
	\begin{eqnarray}
	u(\tau_x\omega, b)-u(\tau_x\omega,0)&=& \mathbb{E}_{\tilde{\mu}}(\phi_b-\phi_0|\mathscr{E})(\tau_x\omega)\nonumber\\&=&\mathbb{E}_{\tilde{\mu}}(\phi_{x+b}-\phi_x|\mathscr{E})(\omega)\nonumber\\
	&=& u(\omega, x+b)-u(\omega,x)
	\end{eqnarray}
	and so $u$ is shift-covariant, as defined in (\ref{shiftcov1}) and (\ref{shiftcov2}).
	
	Thirdly, by Theorem \ref{mymain1}, for $p = 1$ and $2$
	\begin{equation}
	\mathbb{E}_{\tilde{\mu}}\left(\left(\phi_x-\phi_0\right)^{2p}\right) < \infty.
	\end{equation}
	Then
	\begin{eqnarray}
	\mathbb{E}_{\tilde{\mu}} |u(\omega,x)|^2& =& \mathbb{E}_{\tilde{\mu}}\left(\mathbb{E}_{\tilde{\mu}}(\phi_x-\phi_0|\mathscr{E})\right)^2\nonumber\\ 
	&\le& \mathbb{E}_{\tilde{\mu}}\left(\mathbb{E}_{\tilde{\mu}}\left((\phi_x-\phi_0)^2|\mathscr{E}\right)\right)\nonumber\\
	& =& \mathbb{E}_{\tilde{\mu}}\left((\phi_x-\phi_0)^2\right)\nonumber <\infty, 
	\end{eqnarray}
	where the second line is given by Jensen inequality with respect to conditional expectation. This gives the square integrability for each component. As for the square integrability as a vector field, by Corollary \ref{pbound}, we have $\mathbb{E}_{\tilde{\mu}}(\omega_{0x}^2) < \infty$. Then
	\begin{eqnarray}
	\mathbb{E}_{\tilde{\mu}} \left(\sum_{|x| = 1}\omega_{0x}|u(\omega,x)|^2\right) & = & \sum_{|x| = 1}\mathbb{E}_{\tilde{\mu}}\omega_{0x}|u(\omega,x)|^2\nonumber \\& \le&  \sum_{|x| = 1}\sqrt{\mathbb{E}_{\tilde{\mu}}\omega^2_{0x}\mathbb{E}_{\tilde{\mu}}|u(\omega,x)|^4}\nonumber\\
	& =&  \sum_{|x| = 1}\sqrt{\mathbb{E}_{\tilde{\mu}}\omega^2_{0x}\mathbb{E}_{\tilde{\mu}}\left(\mathbb{E}_{\tilde{\mu}}(\phi_x-\phi_0|\mathscr{E})\right)^4}\nonumber\\
	& \le&  \sum_{|x| = 1}\sqrt{\mathbb{E}_{\tilde{\mu}}\omega^2_{0x}\mathbb{E}_{\tilde{\mu}}\left(\mathbb{E}_{\tilde{\mu}}((\phi_x-\phi_0)^4|\mathscr{E})\right)}\nonumber\\
	& =&  \sum_{|x| = 1}\sqrt{\mathbb{E}_{\tilde{\mu}}\omega^2_{0x}\mathbb{E}_{\tilde{\mu}}\left((\phi_x-\phi_0)^4\right)}<\infty.
	\end{eqnarray}
	This gives the square integrability as a vector field.
	
	Finally, the definition of $u$ and the fact that $\tilde{\mu}$ is translation invariant imply that
	\begin{equation}
	\mathbb{E}_{\tilde{\mu}}(u(\cdot,x)) = \mathbb{E}_{\tilde{\mu}}(\phi_x-\phi_0) = 0.
	\end{equation}
	As $u$ obeys all conditions of Lemma \ref{lemmazerofunction}, we have 
	$\mathbb{E}_{\tilde{\mu}}(\phi_x-\phi_0|\mathscr{E})(\omega)  = 0$, $\tilde{\mu}-a.s$, namely
	\begin{equation}
	\mathbb{E}_{\tilde{\mu}}(\phi_x|\mathscr{E})(\omega)  = \mathbb{E}_{\tilde{\mu}}(\phi_0|\mathscr{E})(\omega), \quad\tilde{\mu}-a.s
	\end{equation}
	
	Furthermore, consider $h(\omega) = \mathbb{E}_{\tilde{\mu}}(\phi_0|\mathscr{E})(\omega)$ as a function of $\omega$. Then for $x \in \mathbb{Z}^d$
	\begin{equation}
	h(\tau_x \omega) = \mathbb{E}_{\tilde{\mu}}(\phi_0|\mathscr{E})(\tau_x\omega) = \mathbb{E}_{\tilde{\mu}}(\phi_x|\mathscr{E})(\omega) = h(\omega).
	\end{equation}
	By ergodicity of $\tilde{\mu}$, $f(\omega)$ is constant and thus
	\begin{equation}
	\mathbb{E}_{\tilde{\mu}}(\phi_0|\mathscr{E})(\omega) = \mathbb{E}_{\tilde{\mu}}\left(\mathbb{E}_{\tilde{\mu}}(\phi_0|\mathscr{E})(\omega)\right) = \mathbb{E}_{\tilde{\mu}}(\phi_0).
	\end{equation} 
\end{proof}

\subsection{Regularity estimates}
\hspace*{\parindent}Now we will give the proof of Lemma \ref{reglemma}. which implies that the random operator $\phi$ given in (\ref{defphioperator}) is $L^2$ continuous. With Lemma \ref{reglemma}, we need only to work with smooth and compactly supported test function.
\begin{proof}[Proof of Lemma \ref{reglemma}]
	For any $f \in \mathcal{H}_0$, define $v_f \in \mathbb{R}^{\mathbb{Z}^d}$ by
	\begin{equation}
	v_f(x) = \int dy f(y)\mathbbm{1}_{\lfloor y\rfloor = x}.
	\end{equation}
	Then if $v_f \cdot \phi = \sum v_f(x)\phi(x)$ denotes the usual scalar product, $v_f \cdot \phi = \phi(f)$. The sum over $x$ is convergent and $v_f$ is compact supported and orthogonal to constants as $f \in\mathcal{H}_0$
	
	Let $\Delta_d$ be the lattice Laplacian defined in (\ref{defgradientsquare}). Then by Theorem \ref{mymain1}, $\norm{\phi(f)}_{L^2(\mu)} = \mathbb{E}_\mu((v_f \cdot \phi)^2) \le C_1 [v_f, -\Delta_d^{-1} v_f]$. By (4.4) in \cite{BS11}, $[v_f, -\Delta_d^{-1} v_f] \le C_2 \norm{f}^2_\mathcal{H}$. These two inequalities give us (\ref{reginequal}).
\end{proof}

For convenience of notation, whenever $\mathcal{R}$ is an operator on $\ell^2(\mathbb{Z}^d)$, we will extend it to a bounded operator on $L^2(\mathbb{R}^d)$ via formula
\begin{equation}\label{exttocont}
(f, \mathcal{R}f):=\int dx dy f(x)f(y)\mathcal{R}(\floor*{x}, \floor*{y}),
\end{equation}
where $\mathcal{R}(x, y)$ is the kernel of $\mathcal{R}$ in the canonical basis in $\ell^2(\mathbb{Z}^d)$.
Next we will prove that the tail of $(f, e^{t\mathcal{L}_X^\omega}f)$ is integrable. 

\begin{lemma}\label{lemmatailbd}
	Let $\mu$ be a translation-invariant, ergodic gradient Gibbs measure
	 and $\tilde{\mu}$ is the extension of $\mu$ to $\mathbb{R}^{\mathbb{Z}^d} \times \mathbb{R}^E$ defined in (\ref{muext}). Then for $f \in C_0^\infty(\mathbb{R}^d)$ there exist $N(\omega, f)$ for $\tilde{\mu}-a.s.~\omega$ such that if $t > N(\omega, f)$
	\begin{equation}
	(f, e^{t\mathcal{L}_X^\omega}f)\le C \norm{f}^2_\infty\lambda(\text{supp}~f)^2 \frac{1}{t^{d/2}}
	\end{equation}
	where $\lambda(A)$ is the Lebesgue measure of set $A$.
\end{lemma}
\begin{proof}
	Recall that $X = \{X_t: t \ge 0\}$ is a reversible continuous time Markov chain with generator $\mathcal{L}^\omega_X$ and that $p^\omega (t, x, y)$ be the transition density of $X$ with respect to the reversible measure or the heat kernels associated with $\mathcal{L}^\omega_X$, i.e. $p^\omega (t, x, y) = P^{\omega}_x[X_t = y].$
	Then by (\ref{exttocont})
	\begin{eqnarray}
	&&(f, e^{t\mathcal{L}_X^\omega}f)\nonumber\\
	&&=\sum_{x,y \in \mathbb{Z}^d}\int_{[0,1]^d}dz\int_{[0,1]^d}dz'f(x+z)f(y+z')p^\omega(t,x,y).
	\end{eqnarray}
	
	For any $x \in \text{supp}~f$, by Lemma \ref{lemmaassumptiontest} and Corollary \ref{heatcol}, there exists a constant $C$ and $N(x, \omega)$ s.t. for $\sqrt{t} \ge N(x, \omega)$ and all $y \in \mathbb{Z}^d$,
	\begin{equation}
	p^{\omega}(t, x, y) \le Ct^{-d/2}.
	\end{equation}
	
	Choose $N(f, \omega) = \max_{x \in \text{supp} ~ f}\{N^2(x, \omega)\}$, and then when $t \ge N(f, \omega)$
	\begin{eqnarray}
	&&(f, e^{t\mathcal{L}_X^\omega}f)\nonumber\\
	&&\le\sum_{x,y \in \mathbb{Z}^d}\int_{[0,1]^d}dz\int_{[0,1]^d}dz'f(x+z)f(y+z')Ct^{-d/2}\nonumber\\
	&&\le\int_{[0,1]^d}dz\int_{[0,1]^d}dz'\norm{f}^2_\infty Ct^{-d/2}\sum_{x,y \in \mathbb{Z}^d}\mathbbm{1}_{(y+z' \in supp f)}\mathbbm{1}_{(x+z \in supp f)}\nonumber\\
	&&=C \norm{f}^2_\infty\lambda(\text{supp}~f)^2 \frac{1}{t^{d/2}},
	\end{eqnarray} 
	in which $\lambda(\text{supp}~f)$ by our definition is the Lebesgue measure of $\text{supp}~f$
\end{proof}

\begin{col}\label{trunc}
	Let $\mu$ be a translation-invariant, ergodic gradient Gibbs measure defined in (\ref{DLR}) and $\tilde{\mu}$ is the extension of $\mu$ to $\mathbb{R}^{\mathbb{Z}^d} \times \mathbb{R}^E$ defined in (\ref{muext}). Then for $\tilde{\mu}-a.s.~\omega$, if $f \in C_0^\infty(\mathbb{R}^d)$,
	\begin{equation}\label{tailbound}
	\lim\limits_{M \to \infty}\sup\limits_{0<\delta<1}\int_{M}^{\infty}dt\delta^{-2}(f_\delta, e^{t\delta^{-2}\mathcal{L}_X^\omega}f_\delta) = 0.
	\end{equation}
\end{col}

\begin{proof}
	Recall that $f_\delta(x) = \delta^{d/2+1}f(\delta x)$. Thus  $\norm{f_\delta}_\infty =  \delta^{d/2+1}\norm{f}_\infty$. Also notice that $\lambda(\text{supp}~ f_\delta) = \delta^{-d}\lambda(\text{supp}~ f)$. For $\tilde{\mu}-a.s.~\omega$, let $N(f, \omega)$ be the constant in Lemma \ref{lemmatailbd}. If $t > N(\omega)$ and $0<\delta<1$, $\delta^{-2}t > N(\omega)$. Thus by Lemma \ref{lemmatailbd}, when $t > N(\omega)$, 
	\begin{equation}
	\delta^{-2}(f_\delta, e^{t\delta^{-2}\mathcal{L}_X^\omega}f_\delta) \le \delta^{-2} C \norm{f_\delta}^2_\infty\lambda(\text{supp}~f_\delta)^2 \frac{1}{t^{d/2}} =  C \norm{f}^2_\infty\lambda(\text{supp}~f)^2 \frac{1}{t^{d/2}}.
	\end{equation}
	Then (\ref{tailbound}) holds for $d \ge 3$.
\end{proof}

\subsection{Scaling limit}
\hspace*{\parindent}Before proving the final result about scaling limits, we need two lemmas.
\begin{lemma}
	Let $\mu$ be a translation-invariant, ergodic gradient Gibbs measure and $\tilde{\mu}$ be the extension of $\mu$ to $\mathbb{R}^{\mathbb{Z}^d} \times \mathbb{R}^E$ defined in (\ref{muext}). There exists a positive semidefinite, non-degenerate $d\times d$ matrix $q$, s.t. for every $t > 0$ and $f \in C_0^\infty(\mathbb{R}^d)\cap\mathcal{H}$,
	\begin{equation}
	\delta^{-2}(f_\delta, e^{t\delta^{-2}\mathcal{L}_X^\omega}f_\delta) \to (f, e^{tQ}f)~~~~~~~~~~~~\tilde{\mu}-a.s.~\omega.
	\end{equation}
	where Q is defined from q by (\ref{colmatrix}).
\end{lemma}
\begin{proof}
	By the definition of $\mathcal{L}^\omega_X$ and (\ref{exttocont}), we have
	\begin{eqnarray}
	(f, e^{t\mathcal{L}^\omega_X} f) &= &\int dx dy f(x) f(y) p^\omega(t,\floor*{x}, \floor*{y}) \nonumber\\
	&=&\int dx f(x)\int_{[0,1]^d}dz\sum_{y \in \mathbb{Z}^d}f(y+z) p^\omega(t,\floor*{x}, \floor*{y})\nonumber\\
	&=&\int dx f(x)\int_{[0, 1]^d}dzE^{\floor*{x}}[f(X_t+z)].
	\end{eqnarray}
	
	Recall that $f_\delta(x) = \delta^{d/2+1}f(\delta x)$, so
	\begin{eqnarray}
	&&\delta^{-2}(f_\delta, e^{t\delta^{-2}\mathcal{L}_X^\omega}f_\delta) \nonumber\\&=&\delta^{-2}\int dx\delta^{(d/2+1)}f(\delta x)\int_{[0,1]^d}dz\delta^{(d/2+1)} E^{\floor*{x}}[f(\delta X_{t\delta^{-2}}+\delta z)]\nonumber\\
	&=& \delta^{-d}\int dx f(\delta x)\int_{[0,1]^d}dz E^{\floor*{x}}[f(\delta X_{t\delta^{-2}}+\delta z)]\nonumber\\
	&=& \int dx f(x)\int_{[0,1]^d}dz E^{\floor*{x/\delta}}[f(\delta X_{t\delta^{-2}}+\delta z)].
	\end{eqnarray} 
	The last equation is by variable change $x \to \delta x$. Then by Corollary \ref{thmnotorigin} and the fact that $f \in C_0^\infty(\mathbb{R}^d)$, we have \begin{equation}
	\delta^{-2}(f_\delta, e^{t\delta^{-2}\mathcal{L}_X^\omega}f_\delta) \to (f, e^{tQ}f)~~~ \tilde{\mu}-a.s.~\omega.
	\end{equation}
\end{proof}

\begin{lemma}\label{varconv}
	Let $\mu$ be a translation-invariant, ergodic gradient Gibbs measure defined in (\ref{DLR}) and $\tilde{\mu}$ is the extension of $\mu$ to $\mathbb{R}^{\mathbb{Z}^d} \times \mathbb{R}^E$ defined in (\ref{muext}). There exists a positive semidefinite, non-degenerate $d\times d$ matrix $q$, s.t. for $f \in C_0^\infty(\mathbb{R}^d)\cap\mathcal{H}$,
	\begin{equation}
	\lim\limits_{\delta \downarrow 0}(f_\delta, (-\mathcal{L}_X^\omega)^{-1} f_\delta) = (f, (-Q)^{-1}f)~~~~~~~~~~~~\tilde{\mu}-a.s.~\omega.
	\end{equation}
	where Q is defined from q by (\ref{colmatrix}).
\end{lemma}
\begin{proof}
	For any $f \in C_0^\infty(\mathbb{R}^d)\cap\mathcal{H}$,
	\begin{equation}
	(f, (-\mathcal{L}^\omega_X)^{-1}f) = \int_{0}^{\infty}dt(f, e^{t\mathcal{L}^\omega_X} f).
	\end{equation}
	Replacing $f$ by $f_\delta$ and $t$ by $\delta^{-2}t$, we obtain
	\begin{equation}
	(f_\delta, (-\mathcal{L}_X^\omega)^{-1} f_\delta) = \int_{0}^{\infty}dt\delta^{-2}(f_\delta, e^{t\delta^{-2}\mathcal{L}_X^\omega}f_\delta).
	\end{equation}
	By the above lemma, $\delta^{-2}(f_\delta, e^{t\delta^{-2}\mathcal{L}_X^\omega}f_\delta) \to (f, e^{tQ}f)~ \tilde{\mu}-a.s.~\omega$; the monotonicity in $t$ and continuity of the limit shows that the convergence is actually uniform on compact intervals by Dini's Lemma. By  Corollary \ref{trunc}, the integral can be truncated to a finite interval and similarly for the integral of the limit. Therefore it follows that
	\begin{equation}
	\lim\limits_{\delta \downarrow 0}(f_\delta, (-\mathcal{L}_X^\omega)^{-1} f_\delta) = \int_{0}^{\infty}dt(f, e^{tQ}f) = (f, -Q^{-1}f).
	\end{equation} 
\end{proof}

Now we will prove the main result about the scaling limit.
\begin{proof}[Proof of Theorem \ref{main}]
	Let $\mu$ be a translation-invariant, ergodic gradient Gibbs measure and $\tilde{\mu}$ is the extension of $\mu$ to $\mathbb{R}^{\mathbb{Z}^d} \times \mathbb{R}^E$ defined in (\ref{muext}). We want to show $\phi(f_\delta)$ converge weakly to a normal random variable with mean zero and variance $(f, (-Q)^{-1}f)$. By Lemma \ref{reginequal}, it suffices to prove this for $f \in C_0^\infty(\mathbb{R}^d)\cap\mathcal{H}$.
	
	By Proposition \ref{Propconditionexp}, we know that $\phi(f)$ is Gaussian conditional on $\omega$. The characteristic function of Gaussian random variable $X$
	\begin{equation}
	E(e^{iX}) = e^{iE(X)-1/2\text{Var}(X)},
	\end{equation}
	shows, by conditional expectation and variance in  Proposition \ref{Propconditionexp}, that
	\begin{equation}\label{chafunc}
	E_{\tilde{\mu}}(e^{i\phi(f)}|\mathcal{F})(\omega) = e^{-1/2(f, (-\mathcal{L}^\omega_X)^{-1}f)}.
	\end{equation}
	By Lemma \ref{varconv}, 	$\lim\nolimits_{\delta \downarrow 0}(f_\delta, (-\mathcal{L}_X^\omega)^{-1} f_\delta) = (f, (-Q)^{-1}f)~\tilde{\mu}-a.s.~\omega.$ Since the right-hand side of (\ref{chafunc}) is a bounded function, Theorem \ref{main} follows by the bounded convergence theorem.
\end{proof}

\appendix
\section*{Appendices}
\section{Stable density} \label{App:stabledensity}
\hspace*{\parindent} Let $f_\alpha(x)$ be the unique positive density  such that
\begin{equation}\label{lpstable1}
\int_{0}^{\infty} e^{-\lambda x}f_\alpha(x) dx = e^{-\lambda^\alpha}.
\end{equation}
By the Theorem \ref{laplacestable} later, $f_\alpha(x)$ is the density of an $\alpha$ stable distribution, whose definition will be given later. In the following part, we first introduce the definition and some important properties of stable distributions. Then the properties of the stable densities will be discussed:  log concavity and tail behavior. A good reference for basic facts of stable density is \cite{Zol86}. 

\subsection{Definition of stable distribution}
\hspace*{\parindent}
Here we give the definition of stable distribution. Let $X, X_1, X_2, \cdots$ be independent random variables with a common distribution $R$ and let $S_n = X_1+X_2+\cdots+X_n$.

\begin{definition}\label{defstable}
	The distribution R is stable if for each n there exist constants $c_n >0$, $\gamma_n$ such that
	\begin{equation}
	S_n \stackrel{\mathrm{d}}{=} c_nX+\gamma_n
	\end{equation}
	and R is not concentrated at one point.
\end{definition}

The next theorem characterizes the norming constant $c_n$ and gives the definition of the characteristic exponent of $R$.
\begin{thm}[{\cite[Theorem 1 in Section 6.1]{Fell71}}]
	The norming constants $c_n$ are of the form $c_n = n^{1/\alpha}$ with $0 < \alpha \le 2$. The constant $\alpha$ will be called the characteristic exponent of $R$. 		
\end{thm}

The next theorem  explains that $f_\alpha(x)$  is a special type of stable density with $0<\alpha< 1$.

\begin{thm}[{\cite[Theorem 1 in Section 13.6]{Fell71}}] \label{laplacestable}
	For fixed $0 < \alpha <1$ the function $\gamma_\alpha(\lambda) = e^{-\lambda^{\alpha}}$ is the Laplace transform of a distribution $G_\alpha$ with the following properties:
	\begin{enumerate}
		\item $G_\alpha$ is stable with $\gamma_n = 0$ and $c_n = n^{1/\alpha}$ in Definition \ref{defstable}.
		\item When $x \to \infty$, 
		\begin{equation} x^\alpha[1-G_{\alpha}(x)] \to \frac{1}{\Gamma(1-\alpha)}. \label{dtbinfty}
		\end{equation}
	\end{enumerate}
\end{thm} 

Except for a few $\alpha$'s, like the case of $\alpha = 0.5$ used in \cite{BS12}, there is no closed form for $f_\alpha$(x). However, it is possible to represent the stable density in an integral form. The following theorem is cited from (2.5.10)\footnote{There is a typing error in the definition of $U_\alpha$ in the reference: first power of $U_\alpha$ should be ${\alpha/(1-\alpha)}$ instead of ${\alpha/(\alpha-1)}$} in \cite{Zol86}.

\begin{thm}\label{thmrpint}
	For $0 < \alpha < 1$, let 
	\begin{equation}\label{Ufunction}
	U_\alpha(\phi) = \left(\frac{\sin\alpha\phi}{\alpha\sin\phi}\right)^{\alpha/(1-\alpha)}\frac{\sin(1-\alpha)\phi}{\alpha\sin\phi},\quad \phi \in (-\pi,\pi)
	\end{equation}
	and
	\[ z(x) = (1-\alpha)(x/\alpha)^{\alpha/(\alpha-1)}.\]
	Then for $x \ge 0$, 
	\begin{equation}\label{rpint}
	f_\alpha(x) = \frac{z(x)^{1/\alpha}}{2\pi (1-\alpha)^{1/\alpha}} \int_{-\pi}^\pi U_\alpha(\phi)\exp\left\{-z(x) U_\alpha(\phi)\right\}d\phi.
	\end{equation}
\end{thm}

\subsection{Log concavity}
\hspace*{\parindent}A non-negative function $f : \mathbb{R}^n \to \mathbb{R}_+$ is logarithmically concave (or log-concave for short) if its domain is a convex set, and if it satisfies the inequality
$$f(\theta x + (1 - \theta) y) \geq f(x)^{\theta} f(y)^{1 - \theta}$$
for all $x,y$ in the domain of $f$ and $0 < \theta < 1$. Log concavity of $f_{1/2}(e^t)$ plays an important role in Brydges-Spencer's proof for the case $\alpha = 1/2$. It provides a Brascamp-Lieb bound (Theorem \ref{BLBthm}) for the second moment of the auxiliary field $t$. In the following paragraphs, we will discuss several definitions about unimodality, which is related closely to log concavity. Then we will give a result about log concavity of $f_\alpha(e^t)$.

The term ``unimodality" originally refers to distributions with a unique mode. The unimodal property is not preserved under addiction or multiplication of random variables so it has been strengthened in the literature according to the following definition. See \cite{BCT97, CT98, IC59} for more details about the history and definition of unimodality.
\begin{definition}\quad
	\begin{enumerate}
		\item \cite{IC59}A real random variable X is said to be \textbf{unimodal} (or quasi-concave) if there exists $a \in \mathbb{R}$ such that the functions $P(X \le x)$ and $P(X > x)$ are convex respectively in $(-\infty, a)$ and $(a, +\infty)$. The number a is called a mode of X.
		\item \cite{IC59}A real variable X is said to be \textbf{strongly unimodal} if the sum $X + Y$ is unimodal for all unimodal variables Y that are independent of X. 
		\item \cite[Definition 3.3]{CT98} A real variable X is said to be \textbf{multiplicative strongly unimodal} if the product $XY$ is unimodal for all unimodal variables Y that are independent of X.
	\end{enumerate}
\end{definition}

Strongly unimodal has been proved to be equivalent to the log concavity of the density \cite{BCT97}, while multiplicative strongly unimodal, which is a more recent concept, also turns out to be related to log concavity \cite{CT98}.

\begin{thm}
	\quad
	\begin{enumerate}
		\item \cite[Theorem 6.1.4]{BCT97} Let $X$ be a random variable. Then $X$ is strongly unimodal if and only if it is absolutely continuous and its probability density function $f_X$ is log concave;
		\item  \cite[Theorem 3.7]{CT98} Let $X$ be a unimodal random variable such that 0 is not a mode of X. Then X is multiplicative strongly unimodal if and only if it is absolutely continuous, with a density $f_X$ satisfying that $\{f_X \not= 0\}$ is an interval contained either in  $[0,\infty)$ or in $(-\infty, 0]$, and $f (e^x )$
		(respectively $f (-e^x ))$ is log-concave on this interval.
	\end{enumerate}
\end{thm}

It has been proved in \cite{IC59} that stable distribution functions are unimodal, so to achieve our goal, the next step is to find out whether a stable distribution function is multiplicative strongly unimodal. The following theorem from \cite{Sim10} gives a condition that a stable distribution is multiplicative strongly unimodal.

\begin{thm}[{\cite[Main Theorem]{IC59}}]
	If $X$ is a random variable with the stable distribution of index $\alpha \in (0, 1)$, then $X$ is multiplicative strongly unimodal if and only if $\alpha \le \frac{1}{2}$.
\end{thm}

From above theorem, we immediately have a criteria for log concavity of $f_\alpha(e^t)$.

\begin{col}\label{flogconcave}
	$f_\alpha(e^t)$ is log concave if and only if $\alpha \le \frac{1}{2}$.
\end{col}

\subsection{Tail behavior}
\hspace*{\parindent} Knowledge of the tail behavior of the density of the auxiliary field, $f_\alpha(e^t)$, will play an essential role in the proof of Theorem \ref{mymain1}. It suffices to discuss the tail behavior of $f_\alpha(x)$ at $+\infty$ and $0$. In the following paragraph, $g(x) \asymp h(x), \mathrm{as~} x \to a$ means that there exist positive constants $c, C$,  and a neighborhood $N$ of $a$, s.t. $ch(x) < g(x) < Ch(x)$ for $x \in N$; $g(x) \sim h(x), \mathrm{as~} x \to a$ means $\lim_{x \to a}g(x)/h(x) = 1$.

For the tail behavior of $f_\alpha(x)$ at $\infty$, the next proposition from Section 1.5 in \cite{NOL15} gives the asymptotic behavior of $f_\alpha(x)$.

\begin{prop}\label{tailinf}
	When $x \to \infty$, $f_\alpha(x) \asymp x^{-\alpha-1}$.
\end{prop}

On the other hand, the next proposition in Lemma of \cite{Haw71}  gives a result for the tail behavior of $f_\alpha(x)$ at $0$.

\begin{prop}\label{tailzero}
	When $x \to 0$, $f_\alpha(x) \asymp x^{-(1/(1-\alpha)+\alpha/(2(1-\alpha))}\exp(-cx^{-\alpha/(1-\alpha)})$
	where $c = c(\alpha) = (1-\alpha)\alpha^{\alpha/(1-\alpha)}.$
\end{prop}

Combing these two propositions with Corollary \ref{flogconcave}, we have the following corollary which states the tail behavior of the first derivative of $\ln f_\alpha(e^{t})$.
\begin{col}\label{tailg}
	Let $\alpha < \frac{1}{2}$ and $g(t) = \frac{d}{dt}\ln f_\alpha(e^{t})$. Then
	\begin{enumerate}
		\item $g(t)$ is decreasing;
		\item $g(t) \ge -\alpha - 1$;
		\item $g(t) \to \infty$ as $t \to -\infty$. 
	\end{enumerate} 
\end{col}
\begin{proof}
	As $f_\alpha(e^{t})$ is log concave, $g(t)$ is deceasing in $t$. This gives the first result.
	
	For the second result, because $g(t)$ is deceasing in $t$, it suffices to prove for the case that $t \to \infty$. By Proposition \ref{tailinf}, there exist some constant $C_1$ such that  $f_\alpha(x) > C_1 x^{-\alpha-1}$ as $t \to \infty$, so when $t$ is large
	\begin{equation}\label{lowerbdf}
	\ln f_\alpha(e^{t}) > \ln C_1 + (-\alpha-1)t.
	\end{equation}
	If there exists some $t_0$ such that $g(t_0) < -\alpha - 1$, then for $\epsilon = (-g(t_0) -\alpha - 1)/{2}$, $g(t) < -\alpha - 1$ for $t > t_0$ as $g(t)$ is deceasing. Then
	\begin{equation}
	\ln f_\alpha(e^{t}) = \int_{t_0}^t g(s)ds + \ln f_\alpha(e^{t_0}) < (-\alpha - 1 - \epsilon)t+\ln f_\alpha(e^{t_0})
	\end{equation}
	when $t > t_0$, which contradicts  (\ref{lowerbdf}).
	
	For the third result, if $g(t) < K$ for some constant $K$, then when $t < 0$
	\begin{equation}\label{upperbdf}
	\ln f_\alpha(e^{t}) = \int_{0}^t g(s)ds + \ln f_\alpha(0) > Kt+\ln f_\alpha(1).
	\end{equation}
	On the other hand, by Proposition \ref{tailzero}, there exists some constant $C_2$ such that  when $x \to 0$, $f_\alpha(x) < C_2 x^{-(1/(1-\alpha)+\alpha/(2(1-\alpha))}\exp(-cx^{-\alpha/(1-\alpha)})$, so when $t \to -\infty$
	\begin{equation}
	\ln f_\alpha(e^{t}) < \ln C_2 -(1/(1-\alpha)+\alpha/(2(1-\alpha))t - ce^{-\alpha/(1-\alpha)t},
	\end{equation}
	which contradicts to (\ref{upperbdf}).
	
\end{proof}

The tail behavior of  $\frac{d^2}{dt^2}\ln f_\alpha(e^{t})$ when $t \to -\infty$ is also important in the proof. Before giving this result, we need two more lemmas for preparation.

\begin{lemma}\label{laplacemethod}
	Let $f, g: (-\pi, \pi) \to \mathbb{R}$ satisfy the following conditions:
	\begin{enumerate}
		\item $f$ and $g$ are positive, analytic and even, with global min at 0;
		\item $f$ is strictly monotone on $(0, \pi)$;
		\item $g = O(f^n)$ as $t \to \pi$ for some $n \in \mathbb{N}$.
	\end{enumerate}
	Write $f_n = f^{(n)}(0)$ and $g_n = g^{(n)}(0)$. Then we have the following approximate formula:
	\begin{equation}
	\int_{-\pi}^{\pi}e^{-Nf(x)}g(x)dx \sim \sqrt{\frac{2\pi}{f_2}}g_0e^{-Nf_0}N^{-1/2}\left(1+(-\frac{f_4}{8f^2_2}+\frac{g_2}{2g_0f_2})N^{-1}+O(N^{-2})\right).
	\end{equation}
\end{lemma} 
\begin{proof}Since $f$ and $g$ are even, $f_{2n+1} =  g_{2n+1} = 0$ for $n \ge 0$. Then
	\begin{equation}\label{laplaceintegral}
	\int_{-\pi}^{\pi}e^{-Nf(x)}g(x)dx
	=e^{-Nf_0}\int_{-\pi}^{\pi}e^{-N(f(x)-f_0)}g(x)dx.
	\end{equation}
	Here we apply saddle point estimation to give an approximation to Laplace integral on the right hand side of (\ref{laplaceintegral}). The result is a corollary of Watson's Lemma (see \cite{Watson1918harmonic} for the original reference), but here we will use a more specific case solved in \cite{miller2006applied}. It follows from (3.15) in \cite{miller2006applied} that
	\begin{equation}\label{wastonapprox}
	\int_{-\pi}^{\pi}e^{-N(f(x)-f_0)}g(x)dx\sim\sqrt{\frac{\pi}{N}}\left(\phi_0+\frac{\phi_1}{4N}+O(N^{-2})\right)
	\end{equation} where
	\begin{equation}\label{watsonpara}
	\phi_0 = g_0\sqrt{\frac{2}{f_2}} \quad\quad\text{and}\quad\quad
	\phi_2 = \left(\frac{2g_2}{f_2}-\frac{g_0f_4}{2f^2_2}\right)\sqrt{\frac{2}{f_2}}.
	\end{equation} 
	
	Apply (\ref{wastonapprox}) and (\ref{watsonpara}) to (\ref{laplaceintegral}), and then we have
	\begin{equation}
	\int_{-\pi}^{\pi}e^{-Nf(x)}g(x)dx \sim  \sqrt{\frac{2\pi}{Nf_2}}g_0e^{-Nf_0}\left(1+(-\frac{f_4}{8f^2_2}+\frac{g_2}{2g_0f_2})N^{-1}+O(N^{-2})\right).
	\end{equation}
\end{proof}

For $U_\alpha$ defined in (\ref{Ufunction}), the next lemma is about the property of $U_\alpha$, which is the summary of  Lemma 2.7.5 and  
the calculation above Theorem 2.5.2 in \cite{Zol86}.
\begin{lemma}\label{Ulemma} For $U_\alpha$ defined above, 
	\begin{enumerate}
		\item $U_\alpha$ are positive, analytic and even, with global minimum at 0;
		\item $U_\alpha$ is strictly monotone on $(0, \pi)$.
	\end{enumerate}
\end{lemma}

Combining above two lemmas, we have the following proposition.

\begin{prop}\label{secondbound}
	Assume $\alpha \le \frac{1}{2}$. Then
	\begin{enumerate}
		\item $\frac{d}{dt}\ln f_\alpha(e^{t}) \sim \alpha^{1/(1-\alpha)}\exp(\frac{\alpha}{\alpha-1}t)+O(1)$ as $t \to -\infty$;
		\item $\frac{d^2}{dt^2}\ln f_\alpha(e^{t}) \to -\infty$ as $t \to -\infty$.
	\end{enumerate}	
\end{prop}
\begin{proof}
	By Theorem \ref{thmrpint}, if $U_\alpha$ is defined above and
	\[ z = (1-\alpha)(e^t/\alpha)^{\alpha/(\alpha-1)},\]
	then
	\begin{equation}
	\ln f_\alpha(e^t) = C+\frac{1}{\alpha-1}t+\ln\left(\int_{-\pi}^\pi U_\alpha(x)\exp\left\{-z U_\alpha(x)\right\}dx\right),
	\end{equation}
	where $C$ is some constant depending only on $\alpha$. Let $$g(t) = \int_{-\pi}^\pi U_\alpha(x)\exp\left\{-z U_\alpha(x)\right\}dx$$. It suffices to prove $\frac{d^2}{dt^2}\ln g_\alpha(e^{t}) \to \infty$ as $t \to -\infty$. Notice that when $t \to -\infty$, $z \to \infty$ and that $\frac{dz}{dt} = cz$ where $c = \frac{\alpha}{\alpha-1}$ . By Lemma \ref{laplacemethod} and \ref{Ulemma}, writing $U_\alpha^{(n)}(0) = u_n$, we have, when $t \to -\infty$,
	\begin{eqnarray}
	g(t) &=& \int_{-\pi}^\pi U_\alpha(x)\exp\left\{-z U_\alpha(x)\right\}dx \nonumber\\
	&\sim&u_0\sqrt{\frac{2\pi}{u_2}}z^{-1/2}\exp\left\{-z u_0)\right\}[1+(-\frac{u_4}{8u^2_2}+\frac{u_2}{2u_0u_2})z^{-1}+O(z^{-2})];\\
	g'(t) &=& -cz\int_{-\pi}^\pi U^2_\alpha(x)\exp\left\{-z U_\alpha(x)\right\}dx \nonumber\\
	&\sim&-czu^2_0\sqrt{\frac{2\pi}{u_2}}z^{-1/2}\exp\left\{-z u_0\right\}[1+(-\frac{u_4}{8u^2_2}+\frac{2u_0u_2}{2u_0^2u_2})z^{-1}+O(z^{-2})];\nonumber\\
	&&\\
	g''(t) &=& c^2z^2\int_{-\pi}^\pi U^3_\alpha(x)\exp\left\{-z U_\alpha(x)\right\}dx - c^2z\int_{-\pi}^\pi U^2_\alpha(x)\exp\left\{-z U_\alpha(x)\right\}dx\nonumber\\
	&\sim&c^2z^2u^3_0\sqrt{\frac{2\pi}{u_2}}z^{-1/2}\exp\left\{-z u_0\right\}[1+(-\frac{u_4}{8u^2_2}+\frac{3u_0^2u_2}{2u_0^3u_2})z^{-1}+O(z^{-2})]\nonumber\\
	&&-c^2zu^2_0\sqrt{\frac{2\pi}{u_2}}z^{-1/2}\exp\left\{-z u_0\right\}[1+(-\frac{u_4}{8u^2_2}+\frac{2u^2_0u_2}{2u_0^3u_2})z^{-1}+O(z^{-2})].\nonumber\\
	&&
	\end{eqnarray}
	Using the fact that $u_0 = 1$, we have
	\begin{eqnarray}
	g'/g &\sim& -cz[1+\frac{3}{2}z^{-1}+O(z^{-2})];\\
	(g'/g)^2 &\sim& c^2z^2[1+3z^{-1}+O(z^{-2})];\\
	g''/g&\sim& c^2z^2[1+z^{-1}+O(z^{-2})]-c^2z(1+O(z^{-1}))\nonumber\\
	&=&c^2z^2[1+O(z^{-2})].
	\end{eqnarray}
	
	For the first derivative of $\ln f_\alpha(e^t)$, when $t \to -\infty$ we have
	\begin{eqnarray}
	\frac{d}{dt}\ln f_\alpha(e^{t}) &=& 1/(\alpha-1)+g'/g \nonumber\\
	&\sim&-cz+O(1).
	\end{eqnarray} 
	As $z = (1-\alpha)(e^t/\alpha)^{\alpha/(\alpha-1)}$ and $c = \alpha/(\alpha-1)$, this gives us the first result.
	
	For the second derivative of $\ln f_\alpha(e^t)$, when $t \to -\infty$, we have
	\begin{equation}
	\frac{d^2}{dt^2}\ln f_\alpha(e^{t}) = g''/g-(g'/g)^2 = -3c^2z+O(1) \to -\infty.
	\end{equation}
	This gives us the second result.
\end{proof}

\section{Random conductance models}
\subsection{Random Walk in Random Environment}\label{sec::rwre}
\hspace*{\parindent}We collect previous results about random walk in random environment and derive the results needed for this articlec . 
Recall that $\mathbb{P}$ is a probability measure on $(\Omega_\omega, \mathscr{F}_\omega) = ((0, \infty)^E, \mathbb{B}((0, \infty))^{\otimes E})$, 
and we write $\mathbb{E}$ to
denote the expectation with respect to $\mathbb{P}$. For a fixed $\omega \in \Omega$, let $P^{\omega}_x$ be the measure associated with VSRW in the environment $\omega$ starting at $x$.

\begin{definition}
	Let $X_t$ be the VSRW associated with the generator $\mathcal{L}_V$ in (\ref{generator}). Set $X^{(n)}_t := \frac{1}{n}X_{n^2t}$, $t \ge 0$. We say that the Quenched Functional CLT (QFCLT) or quenched invariance principle holds for $X$ if there is a matrix $\Sigma$ such that for $\mathbb{P}-a.e. \omega$, under $P^{\omega}_0$, $X^{(n)}$ converges in law to a Brownian motion on $\mathbb{R}^d$ with covariance matrix $\Sigma \Sigma^T$. That is, for every $T > 0$ and every bounded continuous function $F$ on the Skorohod space $D([0, T], \mathbb{R}^d)$, we have that $E^{\omega}_0[F(X^{(n)})] \to E^{BM}_0[F(\Sigma\cdot W)],~ \mathbb{P}-a.s$ with $(W,P^{BM}_0)$ being a Brownian motion started at 0.
\end{definition}

The following are two important assumptions on $\mathbb{P}$.
\begin{assumption}\label{basicassump}
	Assume that $\mathbb{P}$ satisfies the following conditions:
	\begin{enumerate}
		\item[(i)] $\mathbb{P}(0 < \omega(e) < \infty) = 1$ and $\mathbb{E}[\omega (e)] < \infty$ for all $e \in E_d$.
		\item[(ii)] $\mathbb{P}$ is ergodic with respect to translations of $\mathbb{Z}^d$.
	\end{enumerate}
\end{assumption}

With additional moment conditions on the conductances $\omega$, we have following QFCLT for $X$.
\begin{thm}[{\cite[Theorem 1.3]{ADS14}}]\label{THMqfclt}
	Suppose that $d \ge 2$ and Assumption \ref{basicassump} holds. Let $p, q \in (1,\infty]$ be such that $1/p + 1/q < 2/d$ and assume that 
	\begin{equation}
	\mathbb{E}[(\omega(e))^p] < \infty \text{~~~~and~~~~} \mathbb{E}[(1/\omega(e))^q] < \infty
	\end{equation}
	for any $e \in E_d$. Then, the QFCLT holds for X with a deterministic non-degenerate covariance matrix $\Sigma^2$.
\end{thm}

The following corollary gives a more general QFCLT. To state the corollary, for every $T > 0$ and every bounded continuous function $F$ on the Skorohod space $D([0, T], \mathbb{R}^d)$, set $\psi_n(x) = E^{\omega}_{\floor*{nx}}[F(X^{(n)})]$ and $\psi_\infty(x) = E^{BM}_0[F(\Sigma\cdot W + x)]$ with $(W,P^{BM}_0)$ being a Brownian motion started at 0. 

\begin{col}\label{thmnotorigin}
Under the assumption of Theorem \ref{THMqfclt}, we have that $\psi_n(x) \to \psi_\infty(x),~\mathbb{P}-a.s.$
\end{col}

\begin{proof}
Let $\chi$ be the corrector defined in Section 2.1 of \cite{ADS14}. For any $x \in \mathbb{R}^d$, replace the definition of $M_t$ in (7) of \cite{ADS14} by
\begin{equation}
	M_t = X_t -\floor*{x/\epsilon}- \chi(\omega, X_t-\floor*{x/\epsilon}),
\end{equation}
which is a martingale under $P^\omega_{\floor*{x/\epsilon}}$ for $\mathbb{P}-a.e.~\omega$, and then the rest of the proof is the same as that of {\cite[Theorem 1.3]{ADS14}}.
\end{proof}

The second result is the heat kernel estimate for transition density. Recall that the heat kernels associated with $\mathcal{L}^\omega_X$ is defined by
\begin{equation}\label{heatkernel1}
p^\omega (t, x, y) := P^{\omega}_x[X_t = y].
\end{equation} 

To state the upper bounds on $p^\omega(t,x,y)$, we need to introduce the distance $d_\omega$ defined by
\begin{equation}\label{chemicaldist}
d_\omega(x, y):= \inf\limits_\gamma\Big\{\sum_{i = 0}^{l_\gamma-1}1\land\omega(z_i, z_{i+1})^{-1/2}\Big\}
\end{equation}
where the infimum is taken over all paths $\gamma = (z_0, \cdots,z_{l_\gamma})$ connecting $x$ and $y$.

We denote by $\tilde{B}(x, r)$ the closed ball with center $x$ and radius $r$ with respect to $d_\omega$, that is $\tilde{B}(x, r) := \{y \in \mathbb{Z}^d | d_\omega \le r\}$. Notice that for usual graph distance $d$, $d_\omega(x, y) < d(x, y)$ and therefore $B(x, r) \subset \tilde{B}(x, r)$ where $B(x, r)$ the closed ball with center $x$ and radius $r$ with respect to $d$. 

For any non-empty, finite $A \subset \mathbb{Z}^d$ and $p \in [1,\infty)$, we introduce space-averaged $\ell^p$-norms on functions $f : A \mapsto R$ by 
\begin{equation}
\norm{f}_{p, A} = \left(\frac{1}{|A|}\sum_{x \in A}|f(x)|^p\right)^{1/p}
\end{equation}

Recall that $u$ and $v$ defined in (\ref{defuv}). We define for any $x \in \mathbb{Z}^d$
\begin{eqnarray}
\tilde{u}_p(x):=\limsup_{n\to\infty}\norm{u^\omega}_{p, \tilde{B}(x, n)}& \text{and} & \tilde{v}_q(x):=\limsup_{n\to\infty}\norm{v^\omega}_{q, \tilde{B}(x, n)}.
\end{eqnarray}

\begin{assumption}\label{assupheatV}
	There exist $p, q \in (1, +\infty]$ with
	\begin{equation}
	\frac{1}{p-1}+\frac{1}{q} < \frac{2}{d}
	\end{equation}
	such that 
	\begin{eqnarray}
	\tilde{u}_p = \sup_{x \in \mathbb{R}^d}\tilde{u}_p(x) < \infty & \text{and} &\tilde{v}_q = \sup_{x \in \mathbb{R}^d}\tilde{v}_q(x)< \infty.
	\end{eqnarray}
	In particular for every $x \in \mathbb{R}^d$ there exists $\tilde{N}(x, \omega) > 2$ such that
	\begin{eqnarray}
	\sup_{n \ge \tilde{N}(x)}\norm{u^\omega}_{p, \tilde{B}(x, n)} \le 2\tilde{u}_p(x)& \text{and} & \sup_{n \ge \tilde{N}(x)}\norm{v^\omega}_{q, \tilde{B}(x, n)} \le 2\tilde{v}_q(x).
	\end{eqnarray}
\end{assumption}

\begin{thm}[{\cite[Theorem 1.10]{ADS15H}}]\label{heatkernalVSRW}
	Suppose that Assumption \ref{assupheatV} holds. Then, there exist constants $c_i(d, p, q, \tilde{u}_p, \tilde{v}_q)$ such that for any given t and x with $\sqrt{t}\ge\tilde{N}(x, \omega)$ and all $y \in \mathbb{R}^d$ the following hold,
	\begin{itemize}
		\item if $d_\omega(x, y) \le c_5t$ then
		\begin{equation}
		p^\omega(t,x,y)\le c_6t^{-d/2}\exp(-c_7d_\omega(x,y)^2/t).
		\end{equation}
		\item if $d_\omega(x, y) \ge c_5t$ then
		\begin{equation}
		p^\omega(t,x,y)\le c_6t^{-d/2}\exp(-c_8d_\omega(x,y)(1\vee \ln(d_\omega(x,y)/t))).
		\end{equation}
	\end{itemize}
\end{thm}

A straightforward corollary of Theorem \ref{heatkernalVSRW} is as follows. 
\begin{col}\label{heatcol}
	Suppose that Assumption \ref{assupheatV} holds. Then, there exist constants $C(d, p, q, \tilde{u}_p, \tilde{v}_q)$ and $N(x, \omega)$ such that for any given t and x with $\sqrt{t}\ge N(x, \omega)$ and all $y \in \mathbb{R}^d$,
	\begin{equation}
	p^\omega(t,x,y)\le Ct^{-d/2}.
	\end{equation}
\end{col}

Now let $\mu$ be a translation-invariant, ergodic gradient Gibbs measure and $\tilde{\mu}$ is the extension of $\mu$ to $\mathbb{R}^{\mathbb{Z}^d} \times \mathbb{R}^E$ defined in (\ref{muext}).  
\begin{lemma}\label{lemmaassumptiontest}
	Assumption  \ref{assupheatV} holds for $\tilde{\mu}-a.s.~\omega$.
\end{lemma}
\begin{proof}
	Since $B(x, r) \subset \tilde{B}(x, r)$， $\tilde{B}(x, r) \to \mathbb{Z}^d$ as $r \to \infty$.
	
	By Lemma \ref{lemmamuext}, $\tilde{\mu}$ is translation-invariant and ergodic. Define $\omega_i: \mathbb{Z}^d \to \mathbb{R}$ by $\omega_i(x) = \omega(x, x+e_i)$ where $e_i$ is the unit vector at $i$th direction. By Corollary \ref{pbound}, $\mathbb{E}_{\tilde{\mu}}(\omega_i(x)^p)$ exist for all $p \in \mathbb{R}$. 
	
	By the ergodicity of $\tilde{\mu}$, for $\tilde{\mu}-a.s.~\omega$
	\begin{eqnarray}
	\bar{u}^\omega_p(x)&:=&\limsup_{n\to\infty}\norm{u^\omega}_{p, \tilde{B}(x, n)} \nonumber\\&\le&\limsup_{n\to\infty}\sum_{i=1}^{d}\norm{\omega_i}_{p, \tilde{B}(x, n)} \nonumber\\&=& \sum_{i=1}^{d}\mathbb{E}_{\tilde{\mu}}(\omega(0,e_i)^p)^{1/p}. 
	\end{eqnarray}
	 Note that the second inequality is due to Minkowski inequality. Similarly we have
	\begin{equation}
	\bar{v}^\omega_q(x):=\limsup_{n\to\infty}\norm{v^\omega}_{q, \tilde{B}(x, n)}\le\sum_{i=1}^{d}\mathbb{E}_{\tilde{\mu}}(\omega(0,e_i)^{-q})^{1/q}.
	\end{equation}
	Notice that $\bar{u}^\omega_p(x)$ and $\bar{v}^\omega_q(x)$ are uniformly bounded with respect to $x$, so $\bar{u}^\omega_p: = \sup_{x \in \mathbb{R}^d}\tilde{u}^\omega_p(x) $ and $ \bar{v}^\omega_q: = \sup_{x \in \mathbb{R}^d}\tilde{v}^\omega_q(x)$ exist. As $p$ and $q$ are arbitrary here, we can choose them such that they satisfy $\frac{1}{p}+\frac{1}{q} \le \frac{2}{d}$. Thus Assumption \ref{assupheatV} holds for $\tilde{\mu}-a.s.~\omega$.
\end{proof}

\subsection{Potential theory}\label{sec::potential}
The proof of the Lemma \ref{lemmazerofunction} leads us to the study of potential theory for operators depending on a
random environment that fall into the class of random conductance models. We have
borrowed some of the notation and results from the paper of Biskup and Spohn \cite{BS11}. Recall that for a function $g: \Omega_\omega\times\mathbb{Z}^d \to \mathbb{R}$, we say $g$ satisfies the {\itshape shift covariance property} if
\begin{equation}\label{shiftcov1a}
g(\omega, x+b) - g(\omega, x) = g(\tau_x\omega, b),
\end{equation}
where  $x \in \mathbb{Z}^d$ and $b$ is a standard basis vector in $\mathbb{Z}^d$, and 
\begin{equation}\label{shiftcov2a}
g(\omega, 0) = 0.
\end{equation} We also define a space shift by $z \in \mathbb{Z}^d$ to be the map $\tau_z : \Omega_\omega \to \Omega_\omega$
\begin{equation}
(\tau_z\omega)(x, y) := \omega(x+z,y+z), ~~\forall \{x, y\} \in E.
\end{equation}

Consider a translation-invariant $\nu$ probability measure on $\Omega_\omega = \left(\mathbb{R}^+\right)^{\mathbb{Z}^d}$. Let $E_\nu$ be the expectation with respect to $\nu$. We call a function $h: \Omega_\omega \mapsto \mathbb{R}$ local if $h$ depends on $\omega_{xy}$ for only finitely edges $xy$ in $E$. Let $L^2(\nu)$ be the closure of the set of all local functions in the topology induced by the inner product
\begin{equation}
\left<h, g\right>:=E_\nu(h(\omega)g(\omega)).
\end{equation}
Let $\hat{E} := \{\hat{e}_1,..., \hat{e}_d \}$ denote the set of standard basis vectors in $\mathbb{Z}^d$. The translations
by the vectors in $\hat{E}$ induce natural unitary maps $T_1,...,T_d$ on $L^2(\nu)$ defined via
\begin{equation}
(T_jh) := h \circ \tau_{\hat{e}_j} , j = 1,...,d.
\end{equation}
Apart from square integrable functions, we will also need
to work with vector fields, by which we will generally 
mean measurable functions $u:\Omega_\omega \times \hat{E} \to \mathbb{R}$ or $\Omega_\omega \times \hat{E} \to \mathbb{R}^d$, depending on the context. Here $u$ is a $d$-component function and we will show that it extends to the vector field in the usual sense in Lemma \ref{vectorfieldext}. We will sometimes write $u_1,...,u_d$ for $u(\cdot, \hat{e}_1),...,u(\cdot, \hat{e}_d)$; note that these may still be vector-valued. 

While we index vector fields only by the positive standard  basis vectors, in certain situations, it is convenient to have them also defined for the negative coordinate directions via
\begin{equation}
u(\omega,-b) = -u(\tau_{-b}\omega,b),\quad b\in \hat{E}.
\end{equation}

Let $L^2_{vec}(\nu)$ be the set of all vector fields with $(u,u) < \infty$, where $(\cdot,\cdot)$ denotes the inner product
\begin{equation}
(u, v): = E_\nu\left(\sum_{b\in \hat{E}}\omega_bu(\omega,b)\cdot v(\omega,b)\right).
\end{equation}
Examples of such functions are the gradients $\nabla h$ of local functions $h \in L^2(\nu)$ defined component-wise via the formula
\begin{equation}
(\nabla h)_j := T_jh-h.
\end{equation}
We denote by $L^2_\nabla(\nu)$ the closure of the set of gradients of local functions in the
topology induced by the above inner product.
\begin{lemma}[{\cite[Lemma 5.2]{BS11}}]\label{vectorfieldext}
	Let $u \in L^2_\nabla(\nu)$. Then, $u$ satisfies the cycle condition
	\begin{equation}\label{conditioncycle}
	\sum_{j=1}^nu(\tau_{x_j}\omega, x_{j+1}-x_j) = 0
	\end{equation}
	for any finite (nearest-neighbor) cycle
	$(x_0,x_1,...,x_n = x_0)$ on $\mathbb{Z}^d$. In
	particular, there exists a shift-covariant function $\bar{u}: \Omega_\omega\times\mathbb{Z}^d \to \mathbb{R}^d$ such that $u(\omega,b) = \bar{u}(\omega,b)$ for every $b \in \hat{E}$.
\end{lemma}

We will henceforth use the convention of writing $\bar{u}$ for the extension of a shift-covariant
vector field $u \in L^2_{vec}$ to a function on $\mathbb{Z}^d$. Notice that the shift $T_j$ extends
naturally via
\begin{equation}
T_j\bar{u}(\omega,x):=\bar{u}(\tau_{\hat{e}_j}\omega, x)=\overline{T_ju}(\omega,x).
\end{equation}
Next, let us characterize the functions in $(L^2_\nabla)^\perp$.
\begin{lemma}[{\cite[Lemma 5.3]{BS11}}]
	For $u \in L^2_{vec}(\nu)$, let $\mathscr{L}u$ be the function in $L^2(\nu)$ defined by
	\begin{equation}
	(\mathscr{L}u)(\omega):=\sum_{b\in \hat{E}}[\omega_bu(\omega, b)-(\tau_{-b}\omega)_bu(\tau_{-b}\omega, b)],
	\end{equation}
	where $−b$ is the standard basis vector opposite to $b$. We then have
	\begin{equation}
	u \in (L_\nabla^2)^\perp \Leftrightarrow \mathscr{L}u = 0,\quad\nu-a.s.
	\end{equation}
	If u satisfies the cycle condition and $\bar{u}$ is its extension, then $\mathscr{L}u(\tau_x\omega) = \mathscr{L}_X^\omega u(κ,x)$ where $\mathscr{L}_X$ is defined in (\ref{generator}).
\end{lemma}
Clearly, all $u \in L^2_\nabla$ are shift-covariant
and have zero mean. A question which naturally arises is whether every
shift-covariant zero-mean $u$ is in $L^2_\nabla$. (Note that this is analogous to asking whether
every closed differential form is exact.) In \cite{BS11} for the case that weights are compacted bounded away from zero and infinity, the answer to this is in the affirmative. However in our case we need one more condition of $u$.
\begin{thm}\label{thmgradient}
	Suppose $\nu$ is ergodic. Suppose $u \in L^2_{vec}$ obeys
	the cycle condition (\ref{conditioncycle}) and $E_\nu u = 0$. Furthermore, suppose $u$ is square integrable for each component in the sense that $\mathbb{E}_{\nu} |u(\omega,x)|^2 < \infty$ for all $x$ with $|x| = 1$. Then $u \in L^2_\nabla$.
\end{thm}
\begin{proof}
The proof of Theorem 5.4 in \cite{BS11} uses the equivalence that $u \in L^2_{vec}(\nu)$ if and only if all of its components are in $L^2(\nu)$ when $\omega_b$'s is uniformly bounded from below and above. Here, the equivalence is replaced by two independent assumptions
	\begin{enumerate}
		\item $u \in L^2_{vec}$, and
		\item $u$ is square integrable for each component.
	\end{enumerate}
The rest of the proof is the same of that of that of Theorem 5.4 in \cite{BS11}.
\end{proof}

We still have to supply the proof of Lemma \ref{lemmazerofunction}.
\begin{proof}[Proof of Lemma \ref{lemmazerofunction}]
	Since $g$ is square integrable as a vector field, $g \in L^2_{vec}$. As g is shift-covariant, square integrable for each component and has zero expectation, Theorem \ref{thmgradient} implies
	that $g \in L^2_\nabla$. However, $g$ is also harmonic and so, in turn, Lemma 5.3 forces $g \in
	(L^2_{vec})^\perp$. Thus, $g = 0$, as desired.
\end{proof}

\newpage
\bibliography{references}
\bibliographystyle{plain}

\end{document}